\documentclass[10pt,reqno]{amsart}

\usepackage{amsthm, mathrsfs, amsmath, amstext, amsxtra, amsfonts, dsfont, amssymb}
\usepackage{lmodern}
\usepackage[dvipsnames]{xcolor}
\usepackage[colorlinks, linkcolor=blue, citecolor=red, urlcolor=OliveGreen, pagebackref, hypertexnames=false]{hyperref}

\newcommand{\N}{\mathbb N}

\newcommand{\R}{\mathbb R}
\newcommand{\C}{\mathbb C}

\newcommand{\eps}{\epsilon}
\newcommand{\Gact} {\gamma_{\text{c}}}
\newcommand{\Gace} {\gamma_{\emph{c}}}
\newcommand{\alphct} {\alpha_{\text{c}}}
\newcommand{\alphce} {\alpha_{\emph{c}}}
\newcommand{\re}[1]{\mbox{Re} \ #1} 
\newcommand{\im}[1]{\mbox{Im} \ #1} 

\newcommand{\defendproof}{\hfill $\Box$} 

\newtheorem{theorem}{Theorem}[section]
\newtheorem{lem}[theorem]{Lemma} 
\newtheorem{prop}[theorem]{Proposition}
\newtheorem{coro}[theorem]{Corollary} 
\theoremstyle{definition}
\newtheorem{defi}[theorem]{Definition}
\newtheorem{rem}[theorem]{Remark}

\title[Blowup solutions intercritical NL4S]{On Blowup solutions to the focusing intercritical nonlinear fourth-order Schr\"odinger equation} 

\author[V. D. Dinh]{Van Duong Dinh}
\address[V. D. Dinh]{Institut de Math\'ematiques de Toulouse UMR5219, Universit\'e Toulouse CNRS, 31062 Toulouse Cedex 9, France and Department of Mathematics, HCMC University of Pedagogy, 280 An Duong Vuong, Ho Chi Minh, Vietnam}
\email{dinhvan.duong@math.univ-toulouse.fr}


\begin{document}

\maketitle
\begin{abstract}
In this paper we study dynamical properties of blowup solutions to the focusing intercritical (mass-supercritical and energy-subcritical) nonlinear fourth-order Schr\"odinger equation. We firstly establish the profile decomposition of bounded sequences in $\dot{H}^{\Gact} \cap \dot{H}^2$. We also prove a compactness lemma and a variational characterization of ground states related to the equation. As a result, we obtain the $\dot{H}^{\Gact}$-concentration of blowup solutions with bounded $\dot{H}^{\Gact}$-norm and the limiting profile of blowup solutions with critical $\dot{H}^{\Gact}$-norm. 
\end{abstract}


\section{Introduction}
\setcounter{equation}{0}
Consider the Cauchy problem for the focusing intercritical nonlinear fourth-order Schr\"odinger equation
\begin{align}
\left\{
\begin{array}{rcl}
i\partial_t u - \Delta^2 u &=& -|u|^{\alpha} u, \quad \text{on } [0,\infty) \times \R^d, \\
u(0) &=& u_0, 
\end{array}
\right.
\label{intercritical NL4S}
\end{align}
where $u$ is a complex valued function defined on $[0,\infty) \times \R^d$ and $2_\star<\alpha<2^\star$ with
\begin{align}
2_\star:= \frac{8}{d}, \quad 2^\star:= \left\{
\begin{array}{cl}
\infty &\text{if } d=1,2,3,4, \\
\frac{8}{d-4} &\text{if } d\geq 5.
\end{array}
\right.
\end{align}
\indent The fourth-order Schr\"odinger equation was introduced by Karpman \cite{Karpman} and Karpman-Shagalov \cite{KarpmanShagalov} taking into account the role of small fourth-order dispersion terms in the propagation of intense laser beams in a bulk medium with Kerr nonlinearity. Such fourth-order Schr\"odinger equations are of the form
\begin{align}
i\partial_t  - \Delta^2 u + \eps \Delta u = \mu |u|^\alpha u, \quad u(0) = u_0, \label{general NL4S}
\end{align}
where $\eps \in \{0, \pm 1\}$, $\mu\in \{\pm\}$ and $\alpha>0$. The equation $(\ref{intercritical NL4S})$ is a special case of $(\ref{general NL4S})$ with $\eps =0$ and $\mu=-1$. The study of nonlinear fourth-order Schr\"odinger equations $(\ref{general NL4S})$ has attracted a lot of interest in the past several years (see \cite{Pausader}, \cite{Pausadercubic}, \cite{HaoHsiaoWang06}, \cite{HaoHsiaoWang07}, \cite{HuoJia}, \cite{MiaoXuZhao09}, \cite{MiaoXuZhao11}, \cite{MiaoWuZhang}, \cite{Dinhfourt}, \cite{Dinhblowup} and references therein).\newline
\indent The equation $(\ref{intercritical NL4S})$ enjoys the scaling invariance
\[
u_\lambda(t,x):= \lambda^{\frac{4}{\alpha}} u(\lambda^4 t, \lambda x), \quad \lambda >0.
\]
It means that if $u$ solves $(\ref{intercritical NL4S})$, then $u_\lambda$ solves the same equation with intial data $u_\lambda(0, x) = \lambda^{\frac{4}{\alpha}} u_0(\lambda x)$. A direct computation shows
\[
\|u_\lambda(0)\|_{\dot{H}^\gamma} = \lambda^{\gamma +\frac{4}{\alpha}-\frac{d}{2}}\|u_0\|_{\dot{H}^\gamma}.
\]
From this, we define the critical Sobolev exponent
\begin{align}
\Gact:= \frac{d}{2}-\frac{4}{\alpha}. \label{critical sobolev exponent}
\end{align}
We also define the critical Lebesgue exponent 
\begin{align}
\alphct: = \frac{2d}{d-2\Gact} = \frac{d\alpha}{4}. \label{critical lebesgue exponent}
\end{align}
By Sobolev embedding, we have $\dot{H}^{\Gact} \hookrightarrow L^{\alphct}$.
The local well-posedness for $(\ref{intercritical NL4S})$ in Sobolev spaces was studied in \cite{Dinhfract, Dinhfourt} (see also \cite{Pausader} for $H^2$ initial data). It is known that $(\ref{intercritical NL4S})$ is locally well-posed  in $H^\gamma$ for $\gamma \geq \max\{\Gact,0\}$ satisfying for $\alpha>0$ not an even integer, 
\begin{align}
\lceil \gamma \rceil \leq \alpha+1, \label{regularity condition}
\end{align}
where $\lceil \gamma \rceil$ is the smallest integer greater than or equal to $\gamma$. This condition ensures the nonlinearity to have enough regularity. Moreover, the solution enjoys the conservation of mass 
\[
M(u(t)) = \int |u(t,x)|^2 dx = M(u_0),
\]
and $H^2$ solution has conserved energy
\[
E(u(t)) = \frac{1}{2} \int |\Delta u(t,x)|^2 dx -\frac{1}{\alpha+2} \int |u(t,x)|^{\alpha+2} dx = E(u_0).
\]
In the subcritical regime, i.e. $\gamma>\Gact$, the existence time depends only on the $H^\gamma$-norm of initial data. There is also a blowup alternative: if $T$ is the maximal time of existence, then either $T=\infty$ or 
\[
T<\infty, \quad \lim_{t\uparrow T} \|u(t)\|_{H^\gamma} =\infty.
\]
It is well-known (see e.g. \cite{Pausader}) that if $\Gact <0$ or $0<\alpha<\frac{8}{d}$, then $(\ref{intercritical NL4S})$ is globally well-posed in $H^2$. Thus the blowup in $H^2$ may occur only for $\alpha \geq \frac{8}{d}$. Recently, Boulenger-Lenzmann established in \cite{BoulengerLenzmann} blowup criteria for $(\ref{general NL4S})$ with radial data in $H^2$ in the mass-critical ($\Gact=0$), mass and energy intercritical ($0<\Gact<2$) and enery-critical ($\Gact=2$) cases. This naturally leads to the study of dynamical properties of blowup solutions such as blowup rate, concentration and limiting profile, etc. \newline
\indent In the mass-critical case $\Gact=0$ or $\alpha=\frac{8}{d}$, the study of $H^2$ blowup solutions to $(\ref{intercritical NL4S})$ is closely related to the notion of ground states which are solutions of the elliptic equation
\[
\Delta^2Q + Q - |Q|^{\frac{8}{d}}Q =0. 
\]
Fibich-Ilan-Papanicolaou in \cite{FibichIlanPapanicolaou} showed some numerical observations which implies that if $\|u_0\|_{L^2}<\|Q\|_{L^2}$, then the solution exists globally; and if $\|u_0\|_{L^2} \geq \|Q\|_{L^2}$, then the solution may blow up in finite time. Later, Baruch-Fibich-Mandelbaum in \cite{BaruchFibichMandelbaum} proved some dynamical properties such as blowup rate, $L^2$-concentration for radial blowup solutions. In \cite{ZhuYangZhang10}, Zhu-Yang-Zhang established the profile decomposition and a compactness result to study dynamical properties such as $L^2$-concentration, limiting profile with minimal mass of blowup solutions in general case (i.e. without radially symmetric assumption). For dynamical properties of blowup solutions with low regularity initial data, we refer the reader to \cite{ZhuYangZhang11} and \cite{Dinhblowup}. \newline
\indent In the mass and energy intercritical case $0<\Gact<2$, there are few works concerning dynamical properties of blowup solutions to $(\ref{intercritical NL4S})$. To our knowledge, the only paper addressed this problem belongs to \cite{ZhuYangZhang_pc} where the authors studied $L^{\alphct}$-concentration of radial blowup solutions. We also refer to \cite{BaruchFibich} for numerical study of blowup solutions to the equation. \newline
\indent The main purpose of this paper is to show dynamical properties of blowup solutions to $(\ref{intercritical NL4S})$ with initial data in $\dot{H}^{\Gact} \cap \dot{H}^2$. The main difficulty in this consideration is the lack of conservation of mass. To study dynamics of blowup solutions in $\dot{H}^{\Gact} \cap \dot{H}^2$, we firstly need the local well-posedness. For data in $H^2$, the local well-posedness is well-known (see e.g. \cite{Pausader}). However, for data in $\dot{H}^{\Gact} \cap \dot{H}^2$ the local theory is not a trivial consequence of the one for $H^2$ data due to the lack of mass conservation. We thus need to show a new local theory for our purpose, and it will be done in Section $\ref{section preliminaries}$. It is worth noticing that thanks to Strichartz estimates with a ``gain'' of derivatives, we can remove the regularity requirement $(\ref{regularity condition})$. However, we can only show the local well-posedness in dimensions $d\geq 5$, the one for $d\leq 4$ is still open. After the local theory is established, we need to show the existence of blowup solutions. In \cite{BoulengerLenzmann}, the authors showed blowup criteria for radial $H^2$ solutions to $(\ref{general NL4S})$. In their proof, the conservation of mass plays a crucial role. In our setting, the lack of mass conservation makes the problem more difficult. We are only able to prove a blowup criteria for negative energy radial solutions with an additional condition 
\begin{align}
\sup_{t\in [0,T)} \|u(t)\|_{\dot{H}^{\Gact}} <\infty. \label{bounded introduction}
\end{align}
This condition is also needed in our results for dynamical properties of blowup solutions. We refer to Section $\ref{section global existence blowup}$ for more details. To study blowup dynamics for data in $\dot{H}^{\Gact} \cap \dot{H}^2$, we establish the profile decomposition for bounded sequences in $\dot{H}^{\Gact} \cap \dot{H}^2$. This is done by following the argument of \cite{HmidiKeraani} (see also \cite{Guoblowup}). With the help of this profile decomposition, we study the sharp constant to the Gagliardo-Nirenberg inequality
\begin{align}
\|f\|^{\alpha+2}_{L^{\alpha+2}} \leq A_{\text{GN}} \|f\|^{\alpha}_{\dot{H}^{\Gact}} \|f\|^2_{\dot{H}^2}. \label{sharp gagliardo-nirenberg inequality intro}
\end{align}
It follows (see Proposition $\ref{prop variational structure ground state intercritical}$) that the sharp constant $A_{\text{GN}}$ is attained at a function $U \in \dot{H}^{\Gact} \cap \dot{H}^2$ of the form
\[
U(x) = a Q(\lambda x + x_0),
\]
for some $a \in \C^*, \lambda>0$ and $x_0 \in \R^d$, where $Q$ is a solution to the elliptic equation
\begin{align*}
\Delta^2 Q +  (-\Delta)^{\Gact} Q - |Q|^\alpha Q =0. 
\end{align*}
Moreover, 
\[
A_{\text{GN}} = \frac{\alpha+2}{2} \|Q\|^{-\alpha}_{\dot{H}^{\Gace}}.
\]
The profile decomposition also gives a compactness lemma, that is for any bounded sequence $(v_n)_{n\geq 1}$ in $\dot{H}^{\Gace} \cap \dot{H}^2$ satisfying
\[
\limsup_{n\rightarrow \infty} \|v_n\|_{\dot{H}^2} \leq M, \quad \limsup_{n\rightarrow \infty} \|v_n\|_{L^{\alpha+2}} \geq m,
\]
there exists a sequence $(x_n)_{n\geq 1}$ in $\R^d$ such that up to a subsequence,
\[
v_n(\cdot + x_n) \rightharpoonup V \text{ weakly in } \dot{H}^{\Gace} \cap \dot{H}^2,
\]
for some $V \in \dot{H}^{\Gace} \cap \dot{H}^2$ satisfying
\begin{align*}
\|V\|^\alpha_{\dot{H}^{\Gace}} \geq \frac{2}{\alpha+2} \frac{ m^{\alpha+2}}{M^2} \|Q\|_{\dot{H}^{\Gact}}^\alpha. 
\end{align*}
As a consequence, we show that the $\dot{H}^{\Gact}$-norm of blowup solutions satisfying $(\ref{bounded introduction})$ must concentrate by an amount which is bounded from below by $\|Q\|_{\dot{H}^{\Gact}}$ at the blowup time. Finally, we show the limiting profile of blowup solutions with critical norm 
\[
\sup_{t\in [0,T)} \|u(t)\|_{\dot{H}^{\Gact}} = \|Q\|_{\dot{H}^{\Gact}}. 
\]
The plan of this paper is as follows. In Section $\ref{section preliminaries}$, we give some preliminaries including Strichartz estimates, the local well-posednesss for data in $\dot{H}^{\Gact} \cap \dot{H}^2$ and the profile decomposition of bounded sequences in $\dot{H}^{\Gact} \cap \dot{H}^2$. In Section $\ref{section variational analysis}$, we use the profile decomposition to study the sharp Gagliardo-Nirenberg inequality $(\ref{sharp gagliardo-nirenberg inequality intro})$. The global existence and blowup criteria will be given in Section $\ref{section global existence blowup}$. Section $\ref{section blowup concentration}$ is devoted to the blowup concentration, and finally the limiting profile of blowup solutions with critical norm will be given in Section $\ref{section limiting profile}$.
\section{Preliminaries} \label{section preliminaries}
\setcounter{equation}{0}
\subsection{Homogeneous Sobolev spaces}
We firstly recall the definition and properties of homogeneous Sobolev spaces (see e.g. \cite[Appendix]{GinibreVelo}, \cite[Chapter 5]{Triebel} or \cite[Chapter 6]{BerghLofstom}). Given $\gamma \in \R$ and $1 \leq q \leq \infty$, the generalized homogeneous Sobolev space is defined by
\[
\dot{W}^{\gamma,q} := \left\{ u \in \mathcal{S}_0' \ | \ \|u\|_{\dot{W}^{\gamma,q}} := \||\nabla|^\gamma u\|_{L^q} <\infty \right\},
\]
where $\mathcal{S}_0$ is the subspace of the Schwartz space $\mathcal{S}$ consisting of functions $\phi$ satisfying $D^\beta \hat{\phi}(0) =0$ for all $\beta \in \N^d$ with $\hat{\cdot}$ the Fourier transform on $\mathcal{S}$, and $\mathcal{S}_0'$ is its topology dual space. One can see $\mathcal{S}'_0$ as $\mathcal{S}'/\mathcal{P}$ where $\mathcal{P}$ is the set of all polynomials on $\R^d$. Under these settings, $\dot{W}^{\gamma,q}$ are Banach spaces. Moreover, the space $\mathcal{S}_0$ is dense in $\dot{W}^{\gamma,q}$. In this paper, we shall use $\dot{H}^\gamma:= \dot{W}^{\gamma,2}$. We note that the spaces $\dot{H}^{\gamma_1}$ and $\dot{H}^{\gamma_2}$ cannot be compared for the inclusion. Nevertheless, for $\gamma_1 < \gamma < \gamma_2$, the space $\dot{H}^{\gamma}$ is an interpolation space between $\dot{H}^{\gamma_1}$ and $\dot{H}^{\gamma_2}$.
\subsection{Strichartz estimates}
In this subsection, we recall Strichartz estimates for the fourth-order Schr\"odinger equation. Let $I \subset \R$ and $p,q\in [1,\infty]$. We define the mixed norm 
\[
\|u\|_{L^p(I, L^q)} := \Big( \int_I \Big(\int_{\R^d} |u(t,x)|^q dx \Big)^{\frac{p}{q}} \Big)^{\frac{1}{p}},
\]
with a usual modification when either $p$ or $q$ are infinity. We also denote for $(p,q) \in [1,\infty]^2$,
\begin{align}
\gamma_{p,q} = \frac{d}{2}-\frac{d}{q} -\frac{4}{p}. \label{define gamma pq}
\end{align}
\begin{defi}
	A pair $(p,q)$ is called {\bf Schr\"odinger admissible}, for short $(p,q) \in S$, if 
	\[
	(p,q) \in [2,\infty]^2, \quad (p,q,d) \ne (2,\infty, 2), \quad \frac{2}{p} + \frac{d}{q} \leq \frac{d}{2}.
	\]
	A pair $(p,q)$ is call {\bf biharmonic admissible}, for short $(p, q) \in B$, if 
	\[
	(p,q) \in S, \quad \gamma_{p,q} =0.
	\]
\end{defi}
We have the following Strichartz estimates for the fourth-order Schr\"odinger equation.
\begin{prop} [Strichartz estimates \cite{ChoOzawaXia, Dinhfract}] \label{prop generalized strichartz estimate} Let $\gamma \in \R$ and $u$ be a weak solution to the inhomogeneous fourth-order Schr\"odinger equation, namely
	\begin{align}
	u(t) = e^{it\Delta^2} u_0 + i\int_0^t e^{i(t-s) \Delta^2} F(s) ds, \label{weak solution}
	\end{align}
	for some data $u_0$ and $F$. Then for all $(p,q)$ and $(a,b)$ Schr\"odinger admissible with $q<\infty$ and $b<\infty$,
	\begin{align}
	\||\nabla|^\gamma u\|_{L^p(\R, L^q)} \lesssim \||\nabla|^{\gamma+\gamma_{p,q}} u_0\|_{L^2} + \||\nabla|^{\gamma+\gamma_{p,q}-\gamma_{a',b'}-4} F\|_{L^{a'}(\R, L^{b'})}, \label{generalized strichartz estimate}
	\end{align}
	where $(a,a')$ and $(b,b')$ are conjugate pairs. 
\end{prop}
Note that the estimates $(\ref{generalized strichartz estimate})$ are exactly the ones given in \cite{MiaoZhang} or \cite{Pausader} where the authors considered $(p,q)$ and $(a,b)$ are either sharp Schr\"odinger admissible, i.e.
\[
(p,q) \in [2,\infty]^2, \quad (p,q,d) \ne (2, \infty, 2), \quad \frac{2}{p} +\frac{d}{q} =\frac{d}{2},
\]
or biharmonic admissible. We refer to \cite{ChoOzawaXia} or \cite{Dinhfract} for the proof of Propsosition $\ref{prop generalized strichartz estimate}$. Note that instead of using directly a dedicate dispersive estimate of \cite{Ben-ArtziKochSaut} for the fundamental solution of the homogeneous fourth-order Schr\"odinger equation, one uses the scaling technique which is similar to the one of wave equation (see e.g. \cite{KeelTao}). \newline
\indent We also have the following consequence of Strichartz estimates $(\ref{generalized strichartz estimate})$. 
\begin{coro} \label{coro strichartz estimate}
	Let $\gamma \in \R$ and $u$ be a weak solution to the inhomogeneous fourth-order Schr\"odinger equation $(\ref{weak solution})$ for some data $u_0$ and $F$. Then for all $(p,q)$ and $(a,b)$ biharmonic admissible satisfying $q<\infty$ and $b<\infty$,
	\begin{align}
	\|u\|_{L^p(\R, L^q)} \lesssim \|u_0\|_{L^2} + \|F\|_{L^{a'}(\R, L^{b'})}, \label{strichartz estimate L2}
	\end{align}
	and
	\begin{align}
	\||\nabla|^\gamma u\|_{L^p(\R, L^q)} \lesssim \||\nabla|^\gamma u_0\|_{L^2} + \||\nabla|^{\gamma-1} F\|_{L^{2}(\R, L^{\frac{2d}{d+2}})}, \label{strichartz estimate Hgamma}
	\end{align}
\end{coro}
Note that the estimates $(\ref{strichartz estimate Hgamma})$ is important to reduce the regularity requirement of the nonlinearity (see Subsection $\ref{subsection local well posedness}$). \newline
\indent In the sequel, for a space time slab $I\times \R^d$ we define the Strichartz space $\dot{B}^0(I\times \R^d)$ as a closure of $\mathcal{S}_0$ under the norm
\[
\|u\|_{\dot{B}^0(I\times \R^d)}:= \sup_{(p,q) \in B \atop q<\infty} \|u\|_{L^p(I, L^q)}.
\]
For $\gamma \in \R$, the space $\dot{B}^\gamma(I \times \R^d)$ is defined as a closure of $\mathcal{S}_0$ under the norm 
\[
\|u\|_{\dot{B}^\gamma(I\times \R^d)} := \||\nabla|^\gamma u\|_{\dot{B}^0(I \times \R^d)}.
\] 
We also use $\dot{N}^0(I\times \R^d)$  to denote the dual space of $\dot{B}^0(I\times \R^d)$ and 
\[
\dot{N}^\gamma(I \times \R^d):= \left\{u \ : \ |\nabla|^\gamma u \in \dot{N}^0(I \times \R^d)\right\}.
\]
To simplify the notation, we will use $\dot{B}^\gamma(I), \dot{N}^\gamma(I)$ instead of $\dot{B}^\gamma(I\times \R^d)$ and $\dot{N}^\gamma(I \times \R^d)$. By Corollary $\ref{coro strichartz estimate}$, we have
\begin{align}
\|u\|_{\dot{B}^0(\R)} &\lesssim \|u_0\|_{L^2} + \|F\|_{\dot{N}^0(\R)}, \label{strichartz estimate L2 abstract} \\
\|u\|_{\dot{B}^\gamma(\R)} &\lesssim \|u_0\|_{\dot{H}^\gamma} + \||\nabla|^{\gamma-1} F\|_{L^2(\R, L^{\frac{2d}{d+2}})}. \label{strichartz estimate Hgamma abstract}
\end{align}
\subsection{Nonlinear estimates} 
We next recall nonlinear estimates to study the local well-posedness for $(\ref{intercritical NL4S})$. 
\begin{lem}[Nonlinear estimates \cite{Kato}] \label{lem nonlinear estimate}
	Let $F \in C^k(\C, \C)$ with $k \in \N \backslash \{0\}$. Assume that there is $\alpha>0$ such that $k \leq \alpha+1$ and
	\[
	|D^j F(z)| \lesssim |z|^{\alpha+1-j}, \quad z \in \C, j =1, \cdots, k.
	\]
	Then for $\gamma \in [0,k]$ and $1<r, p<\infty$, $1<q \leq \infty$ satisfying $\frac{1}{r} = \frac{1}{p} +\frac{\alpha}{q}$, there exists $C=C(d,\alpha, \gamma, r, p, q)>0$ such that for all $u \in \mathcal{S}$,
	\begin{align}
	\||\nabla|^\gamma F(u)\|_{L^r} \leq C \|u\|^\alpha_{L^q} \||\nabla|^\gamma u\|_{L^p}. \label{nonlinear estimate}
	\end{align}
	Moreover, if $F$ is a homogeneous polynomial in $u$ and $\overline{u}$, then $(\ref{nonlinear estimate})$ holds true for any $\gamma\geq 0$.
\end{lem}
The proof of Lemma $\ref{lem nonlinear estimate}$ is based on the fractional Leibniz rule (or Kato-Ponce inequality) and the fractional chain rule. We refer the reader to \cite[Appendix]{Kato} for the proof.
\subsection{Local well-posedness} \label{subsection local well posedness}
In this subsection, we recall the local well-posedness for $(\ref{intercritical NL4S})$ with initial data in $H^2$ and in $\dot{H}^{\Gact} \cap \dot{H}^2$ respectively. The case in $H^2$ is well-known (see e.g. \cite{Pausader}), while the one in $\dot{H}^{\Gact} \cap \dot{H}^2$ needs a careful consideration. 
\begin{prop} [Local well-posedness in $H^2$ \cite{Pausader}] \label{prop local well-posedness H2}
	Let $d\geq 1$, $u_0 \in H^2$ and $0<\alpha<2^\star$. Then there exist $T>0$ and a unique solution $u$ to $(\ref{intercritical NL4S})$ satisfying
	\[ 
	u \in C([0,T), H^2) \cap L^p_{\emph{loc}}([0,T), W^{2,q}),
	\]
	for any biharmonic admissible pairs $(p,q)$ satisfying $q<\infty$. The time of existence satisfies either $T=\infty$ or $T<\infty$ and $\lim_{t\uparrow T} \|u\|_{\dot{H}^2} =\infty$. Moreover, the solution enjoys the conservation of mass and energy.
\end{prop}
\begin{prop}[Local well-posedness in $\dot{H}^{\Gace} \cap \dot{H}^2$] \label{prop local well-posedness H dot 2}
	Let $d\geq 5$, $2_\star \leq \alpha<2^\star$ and $u_0 \in \dot{H}^{\Gace} \cap \dot{H}^2$. Then there exist $T>0$ and a unique solution $u$ to $(\ref{intercritical NL4S})$ satisfying
	\[
	u \in C([0,T), \dot{H}^{\Gace} \cap \dot{H}^2) \cap L^p_{\emph{loc}}([0,T), \dot{W}^{\Gace, q} \cap \dot{W}^{2,q}),
	\]
	for any biharmonic admissible pairs $(p,q)$ satisfying $q<\infty$. The existence time satisfies either $T=\infty$ or $T<\infty$ and $\lim_{t\uparrow T} \|u(t)\|_{\dot{H}^{\Gace}}+\|u(t)\|_{\dot{H}^2} =\infty$. Moreover, the solution enjoys the conservation of energy. 
\end{prop}
\begin{rem} \label{rem local well-posedness H dot 2}
	\begin{itemize}
		\item When $\Gact=0$ or $\alpha=2_\star$, Proposition $\ref{prop local well-posedness H dot 2}$ is a consequence of Proposition $\ref{prop local well-posedness H2}$ since $\dot{H}^0 = L^2$ and $L^2 \cap \dot{H}^2= H^2$.
		\item In \cite{Dinhfourt}, a similar result holds with an additional regularity assumption $\alpha \geq 1$ if $\alpha$ is not an even integer. Thanks to Strichartz estimate with a ``gain'' of derivatives $(\ref{strichartz estimate Hgamma abstract})$, we can remove this regularity requirement.
	\end{itemize}
\end{rem}
\noindent \textit{Proof of Proposition $\ref{prop local well-posedness H dot 2}$.} We firstly choose
\[
n=\frac{2d}{d+2 - (d-4)\alpha}, \quad n^* = \frac{2d}{d+4-(d-4)\alpha}, \quad m^* = \frac{8}{(d-4)\alpha-4}.
\]
It is easy to check that
\begin{align}
\frac{d+2}{2d}  = \frac{(d-4)\alpha}{2d} + \frac{1}{n}, \quad \frac{1}{n} = \frac{1}{n^*} -\frac{1}{d}, \quad \frac{d}{2} = \frac{4}{m^*} + \frac{d}{n^*}. \label{choice of nmm}
\end{align}
In particular, $(m^*, n^*)$ is a biharmonic admissible and 
\begin{align}
\theta:= \frac{1}{2}-\frac{1}{m^*}=1-\frac{(d-4)\alpha}{8}>0. \label{choice of theta}
\end{align}
Consider 
\[
X:= \left\{u \in  \dot{B}^{\Gact}(I) \cap \dot{B}^2(I) \ : \ \|u\|_{\dot{B}^{\Gact}(I)} + \|u\|_{\dot{B}^2(I)} \leq M \right\},
\]
equipped with the distance
\[
d(u,v) := \|u-v\|_{\dot{B}^0(I)},
\]
where $I=[0,\tau]$ and $M,\tau>0$ to be chosen later. By Duhamel's formula, it suffices to prove that the functional
\[
\Phi(u)(t):= e^{it\Delta^2} u_0 +i \int_0^t e^{i(t-s)\Delta^2} |u(s)|^\alpha u(s) ds
\]
is a contraction on $(X,d)$. By Strichartz estimate $(\ref{strichartz estimate Hgamma abstract})$,
\[
\|\Phi(u)\|_{\dot{B}^2(I)} \lesssim \|u_0\|_{\dot{H}^2}  + \|\nabla(|u|^\alpha u)\|_{L^2(I, L^{\frac{2d}{d+2}})}.
\]
By Lemma $\ref{lem nonlinear estimate}$, 
\[
\|\nabla(|u|^\alpha u)\|_{L^2(I, L^{\frac{2d}{d+2}})} \lesssim \|u\|^\alpha_{L^\infty(I, L^{\frac{2d}{d-4}}) } \|\nabla u\|_{L^2(I, L^n)}.
\]
We next use $(\ref{choice of nmm})$ together with the Sobolev embedding to bound
\begin{align*}
\|u\|_{L^\infty(I, L^{\frac{2d}{d-4}})} \lesssim \|\Delta u\|_{L^\infty(I, L^2)} \lesssim \|u\|_{\dot{B}^2(I)}, 
\end{align*}
and
\begin{align*}
\|\nabla u\|_{L^2(I, L^n)} \lesssim \|\Delta u\|_{L^2(I, L^{n^*})} \lesssim |I|^\theta \|\Delta u\|_{L^{m^*}(I, L^{n^*})} \lesssim |I|^\theta \|u\|_{\dot{B}^2(I)}.
\end{align*}
Thus, we get
\[
\|\Phi(u)\|_{\dot{B}^2(I)} \lesssim \|u_0\|_{\dot{H}^2} + |I|^\theta \|u\|^{\alpha+1}_{\dot{B}^2(I)}.
\]
We now estimate $\|\Phi(u)\|_{\dot{B}^{\Gact}(I)}$. To do so, we separate two cases $\Gact \geq 1$ and $0<\Gact <1$. In the case $\Gact \geq 1$, we estimate as above to get
\[
\|\Phi(u)\|_{\dot{B}^{\Gact}(I)} \lesssim \|u_0\|_{\dot{H}^{\Gact}} + |I|^\theta \|u\|^\alpha_{\dot{B}^2(I)} \|u\|_{\dot{B}^{\Gact}(I)}.
\]
In the case $0<\Gact<1$, we choose 
\begin{align}
p=\frac{8(\alpha+2)}{\alpha(d-4)}, \quad q= \frac{d(\alpha+2)}{d+2\alpha}, \label{choice pq}
\end{align}
and choose $(m,n)$ so that
\[
\frac{1}{p'} = \frac{1}{m} + \frac{\alpha}{p}, \quad \frac{1}{q'}=\frac{1}{q} + \frac{\alpha}{n}.
\]
It is easy to check that $(p,q)$ is biharmonic admissible and $n=\frac{dq}{d-2q}$. The later fact gives the Sobolev embedding $\dot{W}^{2,q} \hookrightarrow L^n$. By Strichartz estimate $(\ref{strichartz estimate L2 abstract})$,
\[
\|\Phi(u)\|_{\dot{B}^{\Gact}(I)} \lesssim \|u_0\|_{\dot{H}^{\Gact}}  + \||\nabla|^{\Gact}(|u|^\alpha u)\|_{L^{p'}(I, L^{q'})}.
\]
By Lemma $\ref{lem nonlinear estimate}$, 
\begin{align*}
\||\nabla|^{\Gact}(|u|^\alpha u)\|_{L^{p'}(I, L^{q'})} & \lesssim \|u\|^\alpha_{L^p(I, L^n)} \||\nabla|^{\Gact}u\|_{L^m(I, L^q)} \\
&\lesssim |I|^{\frac{1}{m}-\frac{1}{p}} \|\Delta u\|^\alpha_{L^p(I, L^q)} \||\nabla|^{\Gact}u\|_{L^p(I, L^q)} \\
& \lesssim |I|^\theta \|u\|^\alpha_{\dot{B}^2(I)} \|u\|_{\dot{B}^{\Gact}(I)}.
\end{align*}
In both cases, we have
\[
\|\Phi(u)\|_{\dot{B}^{\Gact}(I)} \lesssim \|u_0\|_{\dot{H}^{\Gact}} + |I|^\theta \|u\|^\alpha_{\dot{B}^2(I)} \|u\|_{\dot{B}^{\Gact}(I)}.
\]
Therefore,
\[
\|\Phi(u)\|_{\dot{B}^{\Gact}(I) \cap \dot{B}^2(I)} \lesssim \|u_0\|_{\dot{H}^{\Gact} \cap \dot{H}^2} + |I|^\theta \|u\|^\alpha_{\dot{B}^2(I)} \|u\|_{\dot{B}^{\Gact}(I) \cap \dot{B}^2(I)}.
\]
Similarly, by $(\ref{strichartz estimate L2 abstract})$, 
\[
\|\Phi(u)- \Phi(v)\|_{\dot{B}^0(I)} \lesssim \||u|^\alpha u - |v|^\alpha v\|_{L^{p'} (I, L^{q'})},
\]
where $(p,q)$ is as in $(\ref{choice pq})$. We estimate
\begin{align*}
\||u|^\alpha u - |v|^\alpha v\|_{L^{p'}(I, L^{q'})} &\lesssim \left( \|u\|^\alpha_{L^p(I, L^n)} + \|v\|^\alpha_{L^p(I, L^n)}\right) \|u-v\|_{L^m(I, L^q)} \\
&\lesssim |I|^\theta \left(\|\Delta u\|^\alpha_{L^p(I, L^q)} + \|\Delta v\|^\alpha_{L^p(I, L^q)} \right) \|u-v\|_{L^p(I, L^q)} \\
&\lesssim |I|^\theta \left(\|u\|^\alpha_{\dot{B}^2(I)} + \|v\|^\alpha_{\dot{B}^2(I)}\right) \|u-v\|_{\dot{B}^0(I)}.
\end{align*}
This shows that for all $u,v \in X$, there exists $C>0$ independent of $\tau$ and $u_0 \in \dot{H}^{\Gact} \cap \dot{H}^2$ such that
\begin{align}
\|\Phi(u)\|_{\dot{B}^{\Gact}(I)}  + \|\Phi(u)\|_{\dot{B}^2(I)} &\leq C\|u_0\|_{\dot{H}^{\Gact} \cap \dot{H}^2} + C\tau^\theta M^{\alpha+1},  \label{blowup rate proof}\\
d(\Phi(u), \Phi(v)) &\leq C\tau^\theta M^\alpha d(u,v). \nonumber
\end{align}
If we set $M=2C \|u_0\|_{\dot{H}^{\Gact} \cap \dot{H}^2}$ and choose $\tau>0$ so that
\[
C \tau^\theta M^\alpha \leq \frac{1}{2},
\]
then $\Phi$ is a strict contraction on $(X,d)$. This proves the existence of solution 
\[
u \in \dot{B}^{\Gact}(I) \cap \dot{B}^2(I).
\]
The time of existence depends only on the $\dot{H}^{\Gact} \cap \dot{H}^2$-norm of initial data. We thus have the blowup alternative. The conservation of energy follows from the standard approximation. The proof is complete.
\defendproof
\begin{coro}[Blowup rate] \label{coro blowup rate intercritical}
	Let $d \geq 5$, $0<\alpha<2^\star$ and $u_0 \in \dot{H}^{\Gact} \cap \dot{H}^2$. Assume that the corresponding solution $u$ to $(\ref{intercritical NL4S})$ given in Proposition $\ref{prop local well-posedness H dot 2}$ blows up at finite time $0<T<\infty$. Then there exists $C>0$ such that
	\begin{align}
	\|u(t)\|_{\dot{H}^{\Gact} \cap \dot{H}^2} > \frac{C}{(T-t)^{\frac{2-\Gace}{4}}}, \label{blowup rate intercritical}
	\end{align}
	for all $0<t<T$.
\end{coro}
\begin{proof}
	Let $0<t<T$. If we consider $(\ref{intercritical NL4S})$ with initial data $u(t)$, then it follows from $(\ref{blowup rate proof})$ and the fixed point argument that if for some $M>0$,
	\[
	C\|u(t)\|_{\dot{H}^{\Gact} \cap \dot{H}^2} + C(\tau-t)^\theta M^{\alpha+1} \leq M,
	\]
	then $\tau<T$. Thus, 
	\[
	C\|u(t)\|_{\dot{H}^{\Gact} \cap \dot{H}^2} + C(T-t)^\theta M^{\alpha+1} > M,
	\]
	for all $M>0$. Choosing $M= 2C\|u(t)\|_{\dot{H}^{\Gact} \cap \dot{H}^2}$, we see that
	\[
	(T-t)^\theta \|u(t)\|^\alpha_{\dot{H}^{\Gact} \cap \dot{H}^2} >C.
	\]
	This implies
	\[
	\|u(t)\|_{\dot{H}^{\Gact} \cap \dot{H}^2}> \frac{C}{(T-t)^{\frac{\theta}{\alpha}}},
	\]
	which is exactly $(\ref{blowup rate intercritical})$ since $\frac{\theta}{\alpha}= \frac{8-(d-4)\alpha}{8\alpha} = \frac{2-\Gact}{4}$. The proof is complete.
\end{proof}
\subsection{Profile decomposition}
The main purpose of this subsection is to prove the profile decomposition related to the focusing intercritical NL4S by following the argument of \cite{HmidiKeraani} (see also \cite{Guoblowup}).
\begin{theorem}[Profile decomposition] \label{theorem profile decomposition intercritical NL4S}
	Let $d\geq 1$ and $2_\star<\alpha<2^\star$. Let $(v_n)_{n\geq 1}$ be a bounded sequence in $\dot{H}^{\Gace} \cap \dot{H}^2$. Then there exist a subsequence of $(v_n)_{n\geq 1}$ (still denoted $(v_n)_{n\geq 1}$), a family $(x_n^j)_{j\geq 1}$ of sequences in $\R^d$ and a sequence $(V^j)_{j\geq 1}$ of $\dot{H}^{\Gace} \cap \dot{H}^2$ functions such that
	\begin{itemize}
		\item for every $k\ne j$,
		\begin{align}
		|x_n^k - x_n^j| \rightarrow \infty, \quad \text{as } n \rightarrow \infty, \label{pairwise orthogonality intercritical}
		\end{align}
		\item for every $l\geq 1$ and every $x \in \R^d$,
		\[
		v_n(x) = \sum_{j=1}^l V^j(x-x_n^j) + v_n^l(x),
		\]
		with
		\begin{align}
		\limsup_{n\rightarrow \infty} \|v^l_n\|_{L^q} \rightarrow 0, \quad \text{as } l \rightarrow \infty, \label{profile error intercritical}
		\end{align}
		for every $q \in (\alphce,2+2^\star)$, where $\alphce$ is given in $(\ref{critical lebesgue exponent})$. Moreover, 
		\begin{align}
		\|v_n\|^2_{\dot{H}^{\Gace}} &= \sum_{j=1}^l \|V^j\|^2_{\dot{H}^{\Gace}} + \|v^l_n\|^2_{\dot{H}^{\Gace}} + o_n(1), \label{profile identity 1 intercritical}, \\
		\|v_n\|^2_{\dot{H}^2} &= \sum_{j=1}^l \|V^j\|^2_{\dot{H}^2} + \|v^l_n\|^2_{\dot{H}^2} + o_n(1), \label{profile identity 2 intercritical},
		\end{align}
		as $n\rightarrow \infty$.
	\end{itemize}
\end{theorem}
\begin{rem} \label{rem profile decomposition intercritical NL4S}
	In the case $\Gact=0$ or $\alpha=2_\star$, Theorem $\ref{theorem profile decomposition intercritical NL4S}$ is exactly Proposition 2.3 in \cite{ZhuYangZhang10} due to the fact $\dot{H}^0=L^2$ and $L^2 \cap \dot{H}^2 = H^2$.
\end{rem}
\noindent \textit{Proof of Theorem $\ref{theorem profile decomposition intercritical NL4S}$.}
Since $\dot{H}^{\Gact} \cap \dot{H}^2$ is a Hilbert space, we denote $\Omega(v_n)$ the set of functions obtained as weak limits of sequences of the translated $v_n(\cdot + x_n)$ with $(x_n)_{n\geq 1}$ a sequence in $\R^d$. Denote
\[
\eta(v_n):= \sup \{ \|v\|_{\dot{H}^{\Gact}} + \|v\|_{\dot{H}^2} : v \in \Omega(v_n)\}.
\]
Clearly,
\[
\eta(v_n) \leq \limsup_{n\rightarrow \infty} \|v_n\|_{\dot{H}^{\Gact}} + \|v_n\|_{\dot{H}^2}.
\]
We shall prove that there exist a sequence $(V^j)_{j\geq 1}$ of $\Omega(v_n)$ and a family $(x_n^j)_{j\geq 1}$ of sequences in $\R^d$ such that for every $k \ne j$,
\[
|x_n^k - x_n^j| \rightarrow \infty, \quad \text{as } n \rightarrow \infty,
\]
and up to a subsequence, the sequence $(v_n)_{n\geq 1}$ can be written as for every $l\geq 1$ and every $x \in \R^d$,
\[
v_n(x) = \sum_{j=1}^l V^j(x-x_n^j) + v^l_n(x), 
\]
with $\eta(v^l_n) \rightarrow 0$ as $l \rightarrow \infty$. Moreover, the identities $(\ref{profile identity 1 intercritical})$ and $(\ref{profile identity 2 intercritical})$ hold as $n \rightarrow \infty$. \newline
\indent Indeed, if $\eta(v_n) =0$, then we can take $V^j=0$ for all $j\geq 1$. Otherwise we choose $V^1 \in \Omega(v_n)$ such that
\[
\|V^1\|_{\dot{H}^{\Gact}} + \|V^1\|_{\dot{H}^2} \geq \frac{1}{2} \eta(v_n) >0. 
\]
By the definition of $\Omega(v_n)$, there exists  a sequence $(x^1_n)_{n\geq 1} \subset \R^d$ such that up to a subsequence,
\[
v_n(\cdot + x^1_n) \rightharpoonup V^1 \text{ weakly in } \dot{H}^{\Gact} \cap \dot{H}^2.
\]
Set $v_n^1(x):= v_n(x) - V^1(x-x^1_n)$. We see that $v^1_n(\cdot + x^1_n) \rightharpoonup 0$ weakly in $\dot{H}^{\Gact} \cap \dot{H}^2$ and
thus
\begin{align*}
\|v_n\|^2_{\dot{H}^{\Gact}} &= \|V^1\|^2_{\dot{H}^{\Gact}} + \|v^1_n\|^2_{\dot{H}^{\Gact}} + o_n(1),  \\
\|v_n\|^2_{\dot{H}^2} &= \|V^1\|^2_{\dot{H}^2} + \|v^1_n\|^2_{\dot{H}^2} + o_n(1),
\end{align*}
as $n \rightarrow \infty$. We now replace $(v_n)_{n\geq 1}$ by $(v^1_n)_{n\geq 1}$ and repeat the same process. If $\eta(v^1_n) =0$, then we choose $V^j=0$ for all $j \geq 2$. Otherwise there exist $V^2 \in \Omega(v^1_n)$ and a sequence $(x^2_n)_{n\geq 1} \subset \R^d$ such that
\[
\|V^2\|_{\dot{H}^{\Gact}} + \|V^2\|_{\dot{H}^2} \geq \frac{1}{2} \eta(v^1_n)>0,
\]
and
\[
v^1_n(\cdot+x^2_n) \rightharpoonup V^2 \text{ weakly in } \dot{H}^{\Gact} \cap \dot{H}^2.
\]
Set $v^2_n(x) := v^1_n(x) - V^2(x-x^2_n)$. We thus have
$v^2_n(\cdot +x^2_n) \rightharpoonup 0$ weakly in $\dot{H}^{\Gact} \cap \dot{H}^2$ and 
\begin{align*}
\|v^1_n\|^2_{\dot{H}^{\Gact}} & = \|V^2\|^2_{\dot{H}^{\Gact}} + \|v^2_n\|^2_{\dot{H}^{\Gact}} + o_n(1), \\
\|v^1_n\|^2_{\dot{H}^2} &= \|V^2\|^2_{\dot{H}^2} + \|v^2_n\|^2_{\dot{H}^2} + o_n(1),
\end{align*}
as $n \rightarrow \infty$. We claim that 
\[
|x^1_n - x^2_n| \rightarrow \infty, \quad \text{as } n \rightarrow \infty.
\]
In fact, if it is not true, then up to a subsequence, $x^1_n - x^2_n \rightarrow x_0$ as $n \rightarrow \infty$ for some $x_0 \in \R^d$. Since 
\[
v^1_n(x + x^2_n) = v^1_n(x +(x^2_n -x^1_n) + x^1_n),
\]
and $v^1_n (\cdot + x^1_n)$ converges weakly to $0$, we see that $V^2=0$. This implies that $\eta(v^1_n)=0$ and it is a contradiction. An argument of iteration and orthogonal extraction allows us to construct the family $(x^j_n)_{j\geq 1}$ of sequences in $\R^d$ and the sequence $(V^j)_{j\geq 1}$ of $\dot{H}^{\Gact} \cap \dot{H}^2$ functions satisfying the claim above. Furthermore, the convergence of the series $\sum_{j\geq 1}^\infty \|V^j\|^2_{\dot{H}^{\Gact}} + \|V^j\|^2_{\dot{H}^2}$ implies that 
\[
\|V^j\|^2_{\dot{H}^{\Gact}} + \|V^j\|^2_{\dot{H}^2} \rightarrow 0, \quad \text{as } j \rightarrow \infty.
\]
By construction, we have
\[
\eta(v^j_n) \leq 2 \left(\|V^{j+1}\|_{\dot{H}^{\Gact}} + \|V^{j+1}\|_{\dot{H}^2}\right),
\]
which proves that $\eta(v^j_n) \rightarrow 0$ as $j \rightarrow \infty$. To complete the proof of Theorem $\ref{theorem profile decomposition intercritical NL4S}$, it remains to show $(\ref{profile error intercritical})$. To do so, we introduce for $R>1$ a function $\chi_R \in \mathcal{S}$ satisfying $\hat{\chi}_R: \R^d \rightarrow [0,1]$ and
\[
\hat{\chi}_R(\xi) = \left\{
\begin{array}{clc}
1 &\text{if}& 1/R \leq |\xi| \leq R, \\
0 &\text{if}& |\xi| \leq 1/2R \vee |\xi|\geq 2R.
\end{array}
\right.
\]
We write
\[
v^l_n = \chi_R * v^l_n + (\delta - \chi_R) * v^l_n,
\]
where $*$ is the convolution operator. Let $q \in (\alphct, 2+2^\star)$ be fixed. By Sobolev embedding and the Plancherel formula, we have
\begin{align*}
\|(\delta -\chi_R) * v^l_n\|_{L^q} \lesssim \|(\delta-\chi_R) * v^l_n\|_{\dot{H}^\beta} &\lesssim \Big( \int |\xi|^{2\beta} |(1-\hat{\chi}_R(\xi)) \hat{v}^l_n(\xi)|^2 d\xi\Big)^{1/2} \\
&\lesssim \Big( \int_{|\xi|\leq 1/R} |\xi|^{2\beta} |\hat{v}^l_n(\xi)|^2 d\xi\Big)^{1/2} + \Big( \int_{|\xi|\geq R} |\xi|^{2\beta} |\hat{v}^l_n(\xi)|^2 d\xi\Big)^{1/2} \\
&\lesssim R^{\Gact-\beta} \|v^l_n\|_{\dot{H}^{\Gact}} + R^{\beta-2} \|v^l_n\|_{\dot{H}^2},
\end{align*}
where $\beta=\frac{d}{2}-\frac{d}{q} \in (\Gact,2)$. On the other hand, the H\"older interpolation inequality implies
\begin{align*}
\|\chi_R * v^l_n\|_{L^q} &\lesssim \|\chi_R * v^l_n\|^{\frac{\alphct}{q}}_{L^{\alphct}} \|\chi_R * v^l_n\|^{1-\frac{\alphct}{q}}_{L^\infty} \\
&\lesssim \|v^l_n\|^{\frac{\alphct}{q}}_{\dot{H}^{\Gact}} \|\chi_R * v^l_n\|^{1-\frac{\alphct}{q}}_{L^\infty}.
\end{align*}
Observe that
\[
\limsup_{n\rightarrow \infty} \|\chi_R * v^l_n\|_{L^\infty} = \sup_{x_n} \limsup_{n\rightarrow \infty} |\chi_R * v^l_n(x_n)|.
\]
Thus, by the definition of $\Omega(v^l_n)$, we infer that
\[
\limsup_{n\rightarrow \infty} \|\chi_R * v^l_n\|_{L^\infty} \leq \sup \Big\{ \Big| \int \chi_R(-x) v(x) dx\Big| : v \in \Omega(v^l_n)\Big\}.
\]
By the Plancherel formula, we have
\begin{align*}
\Big|\int \chi_R(-x) v(x) dx \Big| &= \Big| \int \hat{\chi}_R(\xi) \hat{v}(\xi) d\xi\Big| \lesssim \|\|\xi|^{-\Gact}\hat{\chi}_R\|_{L^2} \||\xi|^{\Gact}\hat{v}\|_{L^2} \\
&\lesssim R^{\frac{d}{2}-\Gact} \|\hat{\chi}_R\|_{\dot{H}^{-\Gact}} \|v\|_{\dot{H}^{\Gact}} \lesssim R^{\frac{4}{\alpha}} \eta(v^l_n).
\end{align*}
We thus obtain for every $l\geq 1$,
\begin{align*}
\limsup_{n\rightarrow \infty} \|v^l_n\|_{L^q} &\lesssim \limsup_{n\rightarrow \infty} \|(\delta-\chi_R)* v^l_n\|_{L^q} + \limsup_{n\rightarrow \infty} \|\chi_R * v^l_n\|_{L^q} \\
&\lesssim R^{\Gact-\beta} \|v^l_n\|_{\dot{H}^{\Gact}} + R^{\beta-2} \|v^l_n\|_{\dot{H}^2} + \|v^l_n\|^{\frac{\alphct}{q}}_{\dot{H}^{\Gact}} \left[R^{\frac{4}{\alpha}} \eta(v^l_n)\right]^{\left(1-\frac{\alphct}{q}\right)}.
\end{align*}
Choosing $R= \left[\eta(v^l_n)^{-1}\right]^{\frac{\alpha}{4}-\eps}$ for some $\eps>0$ small enough, we see that 
\[
\limsup_{n\rightarrow \infty} \|v^l_n\|_{L^q} \lesssim \eta(v^l_n)^{(\beta-\Gact)\left(\frac{\alpha}{4}-\eps\right)} \|v^l_n\|_{\dot{H}^{\Gact}} + \eta(v^l_n)^{(2-\beta)\left(\frac{\alpha}{4}-\eps\right)} \|v^l_n\|_{\dot{H}^2} +\eta(v^l_n)^{\eps \frac{4}{\alpha} \left(1-\frac{\alphct}{q}\right)} \| v^l_n\|_{\dot{H}^{\Gact}}^{\frac{\alphct}{q}}.
\]
Letting $l \rightarrow \infty$ and using the fact that $\eta(v^l_n) \rightarrow 0$ as $l \rightarrow \infty$ and the uniform boundedness in $\dot{H}^{\Gact} \cap \dot{H}^2$ of $(v^l_n)_{l\geq 1}$, we obtain 
\[
\limsup_{n \rightarrow \infty} \|v^l_n\|_{L^q} \rightarrow 0, \quad \text{as } l \rightarrow \infty.
\]
The proof is complete.
\defendproof 
\section{Variational analysis} \label{section variational analysis}
\setcounter{equation}{0}
Let $d\geq 1$ and $2_\star<\alpha<2^\star$. We consider the variational problems
\begin{align*}
A_{\text{GN}}&:=\max\{H(f): f \in \dot{H}^{\Gact} \cap \dot{H}^2\}, & H(f)&:= \|f\|^{\alpha+2}_{L^{\alpha+2}} \div \left[ \|f\|^{\alpha}_{\dot{H}^{\Gact}} \|f\|^2_{\dot{H}^2} \right], \\
B_{\text{GN}}&:= \max\{K(f): f \in L^{\alphct}\cap \dot{H}^2 \}, & K(f)&:= \|f\|^{\alpha+2}_{L^{\alpha+2}} \div \left[ \|f\|^{\alpha}_{L^{\alphct}} \|f\|^2_{\dot{H}^2} \right].
\end{align*}
Here $A_{\text{GN}}$ and $B_{\text{GN}}$ are respectively sharp constants in the Gagliardo-Nirenberg inequalities
\begin{align*}
\|f\|^{\alpha+2}_{L^{\alpha+2}} &\leq A_{\text{GN}} \|f\|^{\alpha}_{\dot{H}^{\Gact}} \|f\|^2_{\dot{H}^2}, \\
\|f\|^{\alpha+2}_{L^{\alpha+2}} &\leq B_{\text{GN}} \|f\|^{\alpha}_{L^{\alphct}} \|f\|^2_{\dot{H}^2}. 
\end{align*}
Let us start with the following observation.
\begin{lem} \label{lem maximizer intercritical}
	If $g$ and $h$ are maximizers of $H(f)$ and $K(f)$ respectively, then $g$ and $h$ satisfy
	\begin{align}
	A_{\emph{GN}} \|g\|^{\alpha}_{\dot{H}^{\Gact}} \Delta^2 g + \frac{\alpha}{2} A_{\emph{GN}} \|g\|^{\alpha-2}_{\dot{H}^{\Gact}}\|g\|^2_{\dot{H}^2} (-\Delta)^{\Gact}g -\frac{\alpha+2}{2}  |g|^{\alpha} g&=0, \label{maximizer equation intercritical 1} \\
	B_{\emph{GN}} \|h\|^{\alpha}_{L^{\alphce}} \Delta^2 h + \frac{\alpha}{2} B_{\emph{GN}} \|h\|^{\alpha-\alphce}_{L^{\alphce}} \|h\|^2_{\dot{H}^2}|h|^{\alphce-2} h -\frac{\alpha+2}{2}  |h|^{\alpha} h&=0, \label{maximizer equation intercritical 2}
	\end{align}
	respectively.
\end{lem}
\begin{proof}
	If $g$ is a maximizer of $H$ in $\dot{H}^{\Gact} \cap \dot{H}^2$, then $g$ must satisfy the Euler-Lagrange equation
	\[
	\frac{d}{d\eps}\Big|_{\vert \eps=0} H(g+\eps \phi) =0,
	\]
	for all $\phi \in \mathcal{S}_0$. A direct computation shows
	\begin{align*}
	\left.\frac{d}{d\eps}\right|_{\eps=0} \|g+\eps \phi\|^{\alpha+2}_{L^{\alpha+2}} &= (\alpha+2)  \int \re{(|g|^\alpha g \overline{\phi})} dx,\\
	\left.\frac{d}{d\eps}\right|_{\eps=0} \|g+\eps \phi\|^\alpha_{\dot{H}^{\Gact}} &= \alpha \|g\|^{\alpha-2}_{\dot{H}^{\Gact}}\int \re{((-\Delta)^{\Gact} g \overline{\phi})}dx, 
	\end{align*}
	and 
	\begin{align*}
	\left.\frac{d}{d\eps}\right|_{\eps=0} \|g+\eps \phi\|^2_{\dot{H}^2} =2 \int \re{(\Delta^2 g \overline{\phi})} dx.
	\end{align*}
	We thus get
	\begin{align*}
	(\alpha+2) \|g\|^{\alpha}_{\dot{H}^{\Gact}} \|g\|^2_{\dot{H}^2} |g|^\alpha g -  \alpha \|g\|^{\alpha+2}_{L^{\alpha+2}} \|g\|^{\alpha-2}_{\dot{H}^{\Gact}} \|g\|^2_{\dot{H}^2} (-\Delta)^{\Gact} g - 2 \|g\|^{\alpha+2}_{L^{\alpha+2}} \|g\|^\alpha_{\dot{H}^{\Gact}} \Delta^2 g =0.
	\end{align*}
	Dividing by $2\|g\|^{\alpha}_{\dot{H}^{\Gact}} \|g\|^2_{\dot{H}^2}$, we obtain $(\ref{maximizer equation intercritical 1})$. The proof of $(\ref{maximizer equation intercritical 2})$ is similar using the fact that
	\[
	\left.\frac{d}{d\eps}\right|_{\eps=0} \|h+\eps \phi\|^\alpha_{L^{\alphct}} = \alpha \|h\|^{\alpha-\alphct}_{L^{\alphct}} \int \re{(|h|^{\alphct-2} h \overline{\phi})}dx.
	\]
	The proof is complete.
\end{proof}
We next use the profile decomposition given in Theorem $\ref{theorem profile decomposition intercritical NL4S}$ to obtain the following variational structure of the sharp constants $A_{\text{GN}}$ and $B_{\text{GN}}$. 
\begin{prop}[Variational structure of sharp constants] \label{prop variational structure ground state intercritical}
	Let $d\geq 1$ and $2_\star<\alpha<2^\star$.
	\begin{itemize}
		\item The sharp constant $A_{\emph{GN}}$ is attained at a function $U \in \dot{H}^{\Gact} \cap \dot{H}^2$ of the form
		\[
		U(x) = a Q(\lambda x + x_0),
		\]
		for some $a \in \C^*, \lambda>0$ and $x_0 \in \R^d$, where $Q$ is a solution to the elliptic equation
		\begin{align}
		\Delta^2 Q +  (-\Delta)^{\Gact} Q - |Q|^\alpha Q =0. \label{elliptic equation critical sobolev}
		\end{align}
		Moreover, 
		\[
		A_{\emph{GN}} = \frac{\alpha+2}{2} \|Q\|^{-\alpha}_{\dot{H}^{\Gace}}.
		\]
		\item The sharp constant $B_{\emph{GN}}$ is attained at a function $V \in L^{\alphce} \cap \dot{H}^2$ of the form 
		\[
		V(x) = b R(\mu x + y_0),
		\]
		for some $b \in \C^*, \mu>0$ and $y_0 \in \R^d$, where $R$ is a solution to the elliptic equation
		\begin{align}
		\Delta^2 R + |R|^{\alphce-2} R - |R|^\alpha R =0. \label{elliptic equation critical lebesgue}
		\end{align}
		Moreover, 
		\[
		B_{\emph{GN}} = \frac{\alpha+2}{2} \|R\|^{-\alpha}_{L^{\alphce}}.
		\]
	\end{itemize}
\end{prop}
\begin{proof}
	We only give the proof for $A_{\text{GN}}$, the one for $B_{\text{GN}}$ is treated similarly using the Sobolev embedding $\dot{H}^{\Gact} \hookrightarrow L^{\alphct}$. We firstly observe that $H$ is invariant under the scaling
	\[
	f_{\mu, \lambda}(x) := \mu f(\lambda x), \quad \mu, \lambda>0.
	\]
	Indeed, a simple computation shows
	\[
	\|f_{\mu, \lambda} \|^{\alpha+2}_{L^{\alpha+2}} = \mu^{\alpha+2} \lambda^{-d} \|f\|^{\alpha+2}_{L^{\alpha+2}}, \quad \|f_{\mu, \lambda} \|_{\dot{H}^{\Gact}}^\alpha = \mu^\alpha \lambda^{-4} \|f\|^\alpha_{\dot{H}^{\Gact}}, \quad \|f_{\mu, \lambda}\|^2_{\dot{H}^2} = \mu^2 \lambda^{4-d} \|f\|_{\dot{H}^2}^2.
	\]
	We thus get $H(f_{\mu, \lambda}) = H(f)$ for any $\mu, \lambda>0$. Moreover, if we set $g(x) = \mu f(\lambda x)$ with 
	\[
	\mu = \left(\frac{\|f\|^{\frac{d}{2}-2}_{\dot{H}^{\Gact}}}{\|f\|_{\dot{H}^2}^{\frac{4}{\alpha}}}\right)^{\frac{1}{2-\Gact}}, \quad \lambda = \left(\frac{\|f\|_{\dot{H}^{\Gact}}}{\|f\|_{\dot{H}^2}}\right)^{\frac{1}{2-\Gact}},
	\]
	then $\|g\|_{\dot{H}^{\Gact}} = \|g\|_{\dot{H}^2} = 1$ and $H(g) = H(f)$. Now let $(v_n)_{n\geq 1}$ be the maximizing sequence such that $H(v_n) \rightarrow A_{\text{GN}}$ as $n\rightarrow \infty$. After scaling, we may assume that $\|v_n\|_{\dot{H}^{\Gact}} = \|v_n\|_{\dot{H}^2} =1$ and $H(v_n) = \|v_n\|^{\alpha+2}_{L^{\alpha+2}} \rightarrow A_{\text{GN}}$ as $n\rightarrow \infty$. Since $(v_n)_{n\geq 1}$ is bounded in $\dot{H}^{\Gact} \cap \dot{H}^2$, it follows from the profile decomposition given in Theorem $\ref{theorem profile decomposition intercritical NL4S}$ that there exist a sequence $(V^j)_{j\geq 1}$ of $\dot{H}^{\Gact} \cap \dot{H}^2$ functions and a family $(x^j_n)_{j\geq 1}$ of sequences in $\R^d$ such that up to a subsequence, 
	\[
	v_n(x) = \sum_{j=1}^l V^j(x-x^j_n) + v^l_n(x),
	\]
	and $(\ref{profile error intercritical})$ and  the identities $(\ref{profile identity 1 intercritical}), (\ref{profile identity 2 intercritical})$ hold. In particular, we have for any $l\geq 1$,
	\begin{align}
	\sum_{j=1}^l \|V^j\|^2_{\dot{H}^{\Gact}} \leq 1, \quad \sum_{j=1}^l \|V^j\|^2_{\dot{H}^2} \leq 1, \label{bounded sequence variational structure intercritical}
	\end{align}
	and
	\[
	\limsup_{n\rightarrow \infty} \|v^l_n\|^{\alpha+2}_{L^{\alpha+2}} \rightarrow 0, \quad \text{as } l \rightarrow \infty. 
	\]
	We have
	\begin{align}
	A_{\text{GN}} &= \lim_{n\rightarrow \infty} \|v_n\|^{\alpha+2}_{L^{\alpha+2}} = \limsup_{n\rightarrow \infty} \Big\| \sum_{j=1}^l V^j(\cdot - x^j_n) + v^l_n \Big\|^{\alpha+2}_{L^{\alpha+2}} \nonumber \\
	&\leq \limsup_{n\rightarrow \infty} \Big( \Big\| \sum_{j=1}^l V^j(\cdot - x^j_n)  \Big\|_{L^{\alpha+2}} + \|v^l_n\|_{L^{\alpha+2}} \Big)^{\alpha+2} \nonumber \\
	&\leq \limsup_{n\rightarrow \infty} \Big\| \sum_{j=1}^\infty V^j(\cdot - x^j_n)  \Big\|_{L^{\alpha+2}}^{\alpha+2}. \label{sum intercritical}
	\end{align}
	By the elementary inequality 
	\begin{align}
	\left| \Big| \sum_{j=1}^l a_j\Big|^{\alpha+2} - \sum_{j=1}^l |a_j|^{\alpha+2} \right| \leq C \sum_{j \ne k} |a_j| |a_k|^{\alpha+1}, \label{elementary inequality intercritical}
	\end{align}
	we have
	\begin{align*}
	\int \Big| \sum_{j=1}^l V^j(x -x^j_n)\Big|^{\alpha+2} dx &\leq \sum_{j=1}^l \int |V^j(x-x^j_n)|^{\alpha+2} dx + C \sum_{j\ne k} \int |V^j(x-x^j_n)||V^k(x-x^k_n)|^{\alpha+1} dx \\
	&\leq \sum_{j=1}^l \int |V^j(x-x^j_n)|^{\alpha+2} dx +  C \sum_{j \ne k} \int |V^j(x+ x^k_n-x^j_n)| |V^k(x)|^{\alpha+1} dx.
	\end{align*}
	Using the pairwise orthogonality $(\ref{pairwise orthogonality intercritical})$, the H\"older inequality implies that $V^j(\cdot + x^k_n-x^j_n) \rightharpoonup 0$ in $\dot{H}^{\Gact} \cap \dot{H}^2$ as $n \rightarrow \infty$ for any $j \ne k$. This leads to the mixed terms in the sum $(\ref{sum intercritical})$ vanish as $n\rightarrow \infty$. This shows that
	\[
	A_{\text{GN}} \leq \sum_{j=1}^\infty \|V^j\|^{\alpha+2}_{L^{\alpha+2}}.
	\]
	By the definition of $A_{\text{GN}}$, we have 
	\[
	\frac{\|V^j\|_{L^{\alpha+2}}^{\alpha+2}}{A_{\text{GN}}} \leq \|V^j\|^{\alpha}_{\dot{H}^{\Gact}} \|V^j\|^2_{\dot{H}^2}.
	\]
	This implies that
	\[
	1 \leq \frac{\sum_{j=1}^\infty \|V^j\|_{L^{\alpha+2}}^{\alpha+2}}{A_{\text{GN}}} \leq \sup_{j\geq 1} \|V^j\|^{\alpha}_{\dot{H}^{\Gact}} \sum_{j=1}^\infty \|V^j\|^2_{\dot{H}^2}.
	\]
	Since $\sum_{j\geq 1} \|V^j\|^2_{\dot{H}^{\Gact}}$ is convergent, there exists $j_0 \geq 1$ such that
	\[
	\|V^{j_0}\|_{\dot{H}^{\Gact}} = \sup_{j\geq 1} \|V^j\|_{\dot{H}^{\Gact}}. 
	\]
	By $(\ref{bounded sequence variational structure intercritical})$, we see that
	\[
	1 \leq \|V^{j_0}\|^{\alpha}_{\dot{H}^{\Gact}} \sum_{j=1}^\infty \|V^j\|_{\dot{H}^2}^2 \leq \|V^{j_0}\|_{\dot{H}^{\Gact}}^{\alpha}. 
	\]
	It follows from $(\ref{bounded sequence variational structure intercritical})$ that $\|V^{j_0}\|_{\dot{H}^{\Gact}}=1$. This shows that there is only one term $V^{j_0}$ is non-zero, hence
	\[
	\|V^{j_0}\|_{\dot{H}^{\Gact}} = \|V^{j_0}\|_{\dot{H}^2} = 1, \quad \|V^{j_0}\|_{L^{\alpha+2}}^{\alpha+2} = A_{\text{GN}}. 
	\]
	It means that $V^{j_0}$ is the maximizer of $H$ and Lemma $\ref{lem maximizer intercritical}$ shows that
	\[
	A_{\text{GN}} \Delta^2 V^{j_0} + \frac{\alpha}{2} A_{\text{GN}}  (-\Delta)^{\Gact} V^{j_0} -\frac{\alpha+2}{2}  |V^{j_0}|^{\alpha} V^{j_0}=0.
	\]
	Now if we set $V^{j_0}(x)= a Q(\lambda x+x_0)$ for some $a \in \C^*, \lambda>0$ and $x_0 \in \R^d$ to be chosen shortly, then $Q$ solves $(\ref{elliptic equation critical sobolev})$ provided that
	\begin{align}
	|a| = \left(\frac{2 \lambda^4 A_{\text{GN}}}{\alpha+2}\right)^{\frac{1}{\alpha}}, \quad \lambda = \Big(\frac{\alpha}{2}\Big)^{\frac{1}{2(2-\Gact)}}. \label{choice of a intercritical}
	\end{align}
	This shows the existence of solutions to the elliptic equation $(\ref{elliptic equation critical sobolev})$. We now compute the sharp constant $A_{\text{GN}}$ in terms of $Q$. We have
	\[
	1=\|V^{j_0}\|^\alpha_{\dot{H}^{\Gact}} = |a|^\alpha \lambda^{-4} \|Q\|^\alpha_{\dot{H}^{\Gact}} = \frac{2A_{\text{GN}}}{\alpha+2} \|Q\|^\alpha_{\dot{H}^{\Gact}}.
	\]
	This implies $A_{\text{GN}} = \frac{\alpha+2}{2} \|Q\|^{-\alpha}_{\dot{H}^{\Gact}}$. The proof is complete. 
\end{proof}
\begin{rem} \label{rem variational analysis intercritical}
	Using $(\ref{choice of a intercritical})$ and the fact
	\begin{align*}
	1 = \|V^{j_0}\|^\alpha_{\dot{H}^{\Gact}} &= |a|^\alpha \lambda^{-4} \|Q\|^\alpha_{\dot{H}^{\Gact}}, \\
	1 = \|V^{j_0}\|^2_{\dot{H}^2} &= |a|^2 \lambda^{4-d} \|Q\|^2_{\dot{H}^2}, \\
	A_{\text{GN}} =\|V^{j_0}\|^{\alpha+2}_{L^{\alpha+2}} &= |a|^{\alpha+2} \lambda^{-d}\|Q\|^{\alpha+2}_{L^{\alpha+2}},
	\end{align*}
	a direct computation shows the following Pohozaev identities
	\begin{align}
	\|Q\|^2_{\dot{H}^{\Gact}} = \frac{\alpha}{2} \|Q\|^2_{\dot{H}^2} = \frac{\alpha}{\alpha+2} \|Q\|^{\alpha+2}_{L^{\alpha+2}}. \label{pohozaev identities}
	\end{align}
	Another way to see above identities is to multiply $(\ref{elliptic equation critical sobolev})$ with $\overline{Q}$ and $x \cdot  \nabla \overline{Q}$ and integrate over $\R^d$ and perform integration by parts. Indeed, multiplying $(\ref{elliptic equation critical sobolev})$ with $\overline{Q}$ and integrating by parts, we get
	\begin{align}
	\|Q\|^2_{\dot{H}^2}  + \|Q\|^2_{\dot{H}^{\Gact}} - \|Q\|^{\alpha+2}_{L^{\alpha+2}} =0. \label{pohozaev equation 1}
	\end{align}
	Multiplying $(\ref{elliptic equation critical sobolev})$ with $x\cdot \nabla \overline{Q}$, integrating by parts and taking the real part, we have
	\begin{align}
	\Big(2-\frac{d}{2}\Big) \|Q\|^2_{\dot{H}^2} + \Big(\Gact-\frac{d}{2}\Big) \|Q\|^2_{\dot{H}^{\Gact}} + \frac{d}{\alpha+2} \|Q\|^{\alpha+2}_{L^{\alpha+2}} =0. \label{pohozaev equation 2}
	\end{align}
	From $(\ref{pohozaev equation 1})$ and $(\ref{pohozaev equation 2})$, we obtain $(\ref{pohozaev identities})$. To see $(\ref{pohozaev equation 2})$, we claim that for $\gamma \geq 0$,
	\begin{align}
	\re{\int (-\Delta)^\gamma Q x \cdot \nabla \overline{Q} dx }  = \Big(\gamma-\frac{d}{2}\Big) \|Q\|^2_{\dot{H}^\gamma}. \label{claim}
	\end{align}
	In fact, by Fourier transform,
	\begin{align}
	\re{\int (-\Delta)^\gamma Q x \cdot \nabla \overline{Q} dx } &= \re{ \int \mathcal{F}[(-\Delta)^\gamma Q] \mathcal{F}^{-1}[x \cdot \nabla \overline{Q}] d\xi} \nonumber \\
	&= \re{ \int \mathcal{F}[(-\Delta)^\gamma Q] \overline{\mathcal{F}[x \cdot \nabla Q]} d\xi} \nonumber \\
	&= \re{ \int |\xi|^{2\gamma} \mathcal{F}(Q) \left(-d \overline{F(Q)} - \xi \cdot \nabla_\xi \overline{\mathcal{F}(Q)} \right) d\xi }  \nonumber \\
	&= -d \|Q\|^2_{\dot{H}^\gamma} - \re{ \int |\xi|^{2\gamma} \mathcal{F}(Q) \xi \cdot \nabla_\xi \overline{F(Q)} d\xi}. \label{pohozaev equation proof}
	\end{align}
	Here we use the fact that $\mathcal{F}(x_j \partial_{x_j} u) = i \partial_{\xi_j} \mathcal{F}(\partial_{x_j} u) = i \partial_{\xi_j} (i\xi_j \mathcal{F}(u)) = -\mathcal{F}(u) - \xi_j \partial_{\xi_j} \mathcal{F}(u)$. By integration by parts, 
	\begin{align*}
	\re{ \int |\xi|^{2\gamma} \mathcal{F}(Q) \xi \cdot \nabla_\xi \overline{\mathcal{F}(Q)} d\xi} = (-2\gamma-d) \|Q\|^2_{\dot{H}^\gamma} - \re{ \int |\xi|^{2\gamma} \xi \cdot \nabla_\xi \mathcal{F}(Q) \overline{\mathcal{F}(Q)} d\xi},
	\end{align*}
	or 
	\[
	\re{ \int |\xi|^{2\gamma} \mathcal{F}(Q) \xi \cdot \nabla_\xi \overline{\mathcal{F}(Q)} d\xi} = \Big(-\gamma -\frac{d}{2} \Big) \|Q\|^2_{\dot{H}^\gamma}.
	\]
	This together with $(\ref{pohozaev equation proof})$ shows $(\ref{claim})$, and $(\ref{pohozaev equation 2})$ follows.
	\newline
	\indent The Pohozaev identities $(\ref{pohozaev identities})$ imply in particular that 
	\[
	H(Q)=\|Q\|^{\alpha+2}_{L^{\alpha+2}} \div \left[\|Q\|^\alpha_{\dot{H}^{\Gact}} \|Q\|^2_{\dot{H}^2} \right] = \frac{\alpha+2}{2} \|Q\|^{-\alpha}_{\dot{H}^{\Gact}} = A_{\text{GN}}, \quad E(Q)=0.
	\] 
	Similarly, we have
	\[
	\|R\|^2_{L^{\alphct}} = \frac{\alpha}{2} \|R\|^2_{\dot{H}^2} = \frac{\alpha}{\alpha+2} \|R\|^{\alpha+2}_{L^{\alpha+2}}.
	\]
	In particular,
	\[
	K(R) = \|R\|^{\alpha+2}_{L^{\alpha+2}} \div \left[ \|R\|^\alpha_{L^{\alphct}} \|R\|^2_{\dot{H}^2} \right] = \frac{\alpha+2}{2} \|R\|^{-\alpha}_{L^{\alphct}} = B_{\text{GN}}, \quad E(R) = 0.
	\]
\end{rem}
\begin{defi}[Ground state] \label{definition ground state}
	\begin{itemize}
		\item We call \textbf{Sobolev ground states} the maximizers of $H$ which are solutions to $(\ref{elliptic equation critical sobolev})$. We denote the set of Sobolev ground states by $\mathcal{G}$. 
		\item We call \textbf{Lebesgue ground states} the maximizers of $K$ which are solutions to $(\ref{elliptic equation critical lebesgue})$. We denote the set of Lebesgue ground states by $\mathcal{H}$. 
	\end{itemize}
\end{defi}
Note that by Lemma $\ref{lem maximizer intercritical}$, if $g, h$ are Sobolev and Lebesgue ground states respectively, then 
\[
A_{\text{GN}} = \frac{\alpha+2}{2} \|g\|^{-\alpha}_{\dot{H}^{\Gact}}, \quad B_{\text{GN}} = \frac{\alpha+2}{2} \|h\|^{-\alpha}_{L^{\alphct}}.
\]
This implies that Sobolev ground states have the same $\dot{H}^{\Gact}$-norm, and all Lebesgue ground states have the same $L^{\alphct}$-norm. Denote
\begin{align}
S_{\text{gs}}&:= \|g\|_{\dot{H}^{\Gact}}, \quad \forall g \in \mathcal{G}, \label{critical sobolev norm} \\
L_{\text{gs}}&:= \|h\|_{L^{\alphct}}, \quad \forall h \in \mathcal{H}. \label{critical lebesgue norm} 
\end{align}
In particular, we have the following sharp Gagliardo-Nirenberg inequalities
\begin{align}
\|f\|^{\alpha+2}_{L^{\alpha+2}} &\leq A_{\text{GN}} \|f\|^{\alpha}_{\dot{H}^{\Gact}} \|f\|^2_{\dot{H}^2}, \label{sharp gagliardo-nirenberg inequality intercritical 1}\\
\|f\|^{\alpha+2}_{L^{\alpha+2}} &\leq B_{\text{GN}} \|f\|^{\alpha}_{L^{\alphct}} \|f\|^2_{\dot{H}^2}, \label{sharp gagliardo-nirenberg inequality intercritical 2}
\end{align}
with 
\[
A_{\text{GN}} = \frac{\alpha+2}{2} S_{\text{gs}}^{-\alpha}, \quad B_{\text{GN}} = \frac{\alpha+2}{2} L_{\text{gs}}^{-\alpha}.
\]
We next give another application of the profile decomposition given in Theorem $\ref{theorem profile decomposition intercritical NL4S}$.
\begin{theorem}[Compactness lemma] \label{theorem compactness lemma intercritical NL4S} Let $d\geq 1$ and $2_\star<\alpha<2^\star$. Let $(v_n)_{n\geq 1}$ be a bounded sequence in $\dot{H}^{\Gace} \cap \dot{H}^2$ such that
	\[
	\limsup_{n\rightarrow \infty} \|v_n\|_{\dot{H}^2} \leq M, \quad \limsup_{n\rightarrow \infty} \|v_n\|_{L^{\alpha+2}} \geq m.
	\]
	\begin{itemize}
		\item Then there exists a sequence $(x_n)_{n\geq 1}$ in $\R^d$ such that up to a subsequence,
		\[
		v_n(\cdot + x_n) \rightharpoonup V \text{ weakly in } \dot{H}^{\Gace} \cap \dot{H}^2,
		\]
		for some $V \in \dot{H}^{\Gace} \cap \dot{H}^2$ satisfying
		\begin{align}
		\|V\|^\alpha_{\dot{H}^{\Gace}} \geq \frac{2}{\alpha+2} \frac{ m^{\alpha+2}}{M^2} S_{\emph{gs}}^\alpha. \label{lower bound critical sobolev}
		\end{align}
		\item Then there exists a sequence $(y_n)_{n\geq 1}$ in $\R^d$ such that up to a subsequence,
		\[
		v_n(\cdot + y_n) \rightharpoonup W \text{ weakly in } L^{\alphce} \cap \dot{H}^2,
		\]
		for some $W \in L^{\alphce} \cap \dot{H}^2$ satisfying
		\begin{align}
		\|W\|^\alpha_{L^{\alphce}} \geq \frac{2}{\alpha+2} \frac{ m^{\alpha+2}}{M^2} L_{\emph{gs}}^\alpha. \label{lower bound critical lebesgue}
		\end{align}
	\end{itemize}
\end{theorem}
\begin{rem} \label{rem compactness lemma intercritical NL4S}
	The lower bounds $(\ref{lower bound critical sobolev})$ and $(\ref{lower bound critical lebesgue})$ are optimal. In fact, if we take $v_n=Q \in \mathcal{G}$ in the first case and $v_n=R \in \mathcal{H}$ in the second case, then we get the equalities. 
\end{rem}
\noindent \textit{Proof of Theorem $\ref{theorem compactness lemma intercritical NL4S}$.} As in the proof of Proposition $\ref{prop variational structure ground state intercritical}$, we only consider the first case, the second case is similar using the Sobolev embedding $\dot{H}^{\Gact} \hookrightarrow L^{\alphct}$. According to Theorem $\ref{theorem profile decomposition intercritical NL4S}$, there exist a sequence $(V^j)_{j\geq 1}$ of $\dot{H}^{\Gact} \cap \dot{H}^2$ functions and a family $(x^j_n)_{j\geq 1}$ of sequences in $\R^d$ such that up to a subsequence, the sequence $(v_n)_{n\geq 1}$ can be written as
\[
v_n(x) = \sum_{j=1}^l V^j(x-x^j_n) + v^l_n(x),
\]
and $(\ref{profile error intercritical})$, $(\ref{profile identity 1 intercritical}), (\ref{profile identity 2 intercritical})$ hold. This implies that
\begin{align}
m^{\alpha+2} &\leq \limsup_{n\rightarrow \infty} \|v_n\|_{L^{\alpha+2}}^{\alpha+2} = \limsup_{n\rightarrow \infty} \Big\| \sum_{j=1}^l V^j(\cdot -x^j_n) + v^l_n\Big\|^{\alpha+2}_{L^{\alpha+2}} \nonumber \\
&\leq \limsup_{n\rightarrow \infty} \Big( \Big\|\sum_{j=1}^l V^j(\cdot -x^j_n) \Big\|_{L^{\alpha+2}} + \|v^l_n\|_{L^{\alpha+2}}\Big)^{\alpha+2} \nonumber \\
&\leq \limsup_{n\rightarrow \infty} \Big\| \sum_{j=1}^\infty V^j(\cdot -x^j_n)\Big\|_{L^{\alpha+2}}^{\alpha+2}. \label{compactness lemma proof intercritical}
\end{align}
By the elementary inequality $(\ref{elementary inequality intercritical})$ and the pairwise orthogonality $(\ref{pairwise orthogonality intercritical})$, the mixed terms in the sum $(\ref{compactness lemma proof intercritical})$ vanish as $n\rightarrow \infty$. We thus get
\[
m^{\alpha+2} \leq \sum_{j=1}^\infty \|V^j\|_{L^{\alpha+2}}^{\alpha+2}.
\]
We next use the sharp Gagliardo-Nirenberg inequality $(\ref{sharp gagliardo-nirenberg inequality intercritical 1})$ to estimate
\begin{align}
\sum_{j=1}^\infty \|V^j\|^{\alpha+2}_{L^{\alpha+2}} \leq \frac{\alpha+2}{2} \frac{1}{S_{\text{gs}}^\alpha} \sup_{j\geq 1} \|V^j\|^{\alpha}_{\dot{H}^{\Gact}} \sum_{j=1}^\infty \|V^j\|^2_{\dot{H}^2}. \label{compactness lemma proof 1 intercritical}
\end{align}
By $(\ref{profile identity 2 intercritical})$, we infer that
\[
\sum_{j=1}^\infty \|V^j\|^2_{\dot{H}^2}  \leq \limsup_{n\rightarrow \infty} \|v_n\|^2_{\dot{H}^2} \leq M^2.
\]
Therefore,
\[
\sup_{j\geq 1} \|V^j\|^{\alpha}_{\dot{H}^{\Gact}} \geq \frac{2}{\alpha+2} \frac{m^{\alpha+2}}{M^2} S_{\text{gs}}^\alpha.
\]
Since the series $\sum_{j\geq 1} \|V^j\|^2_{\dot{H}^{\Gact}}$ is convergent, the supremum above is attained. In particular, there exists $j_0$ such that
\[
\|V^{j_0}\|^{\alpha}_{\dot{H}^{\Gact}} \geq \frac{2}{\alpha+2} \frac{m^{\alpha+2}}{M^2} S_{\text{gs}}^\alpha.
\]
By a change of variables, we write
\[
v_n(x+ x^{j_0}_n) = V^{j_0} (x) + \sum_{1\leq j \leq l \atop j \ne j_0} V^j(x+ x_n^{j_0} - x^j_n) + \tilde{v}^l_n(x),
\]
where $\tilde{v}^l_n(x):= v^l_n(x+x^{j_0}_n)$. The pairwise orthogonality of the family $(x_n^j)_{j\geq 1}$ implies
\[
V^j( \cdot +x^{j_0}_n -x^j_n) \rightharpoonup 0 \text{ weakly in } \dot{H}^{\Gact} \cap \dot{H}^2,
\]
as $n \rightarrow \infty$ for every $j \ne j_0$. We thus get
\begin{align}
v_n(\cdot + x^{j_0}_n) \rightharpoonup V^{j_0} + \tilde{v}^l, \quad \text{as } n \rightarrow \infty, \label{compactness lemma proof 2 intercritical}
\end{align}
where $\tilde{v}^l$ is the weak limit of $(\tilde{v}^l_n)_{n\geq 1}$. On the other hand, 
\[
\|\tilde{v}^l\|_{L^{\alpha+2}} \leq \limsup_{n\rightarrow \infty} \|\tilde{v}^l_n\|_{L^{\alpha+2}} = \limsup_{n\rightarrow \infty} \|v^l_n\|_{L^{\alpha+2}} \rightarrow 0, \quad \text{as } l \rightarrow \infty. 
\]
By the uniqueness of the weak limit $(\ref{compactness lemma proof 2 intercritical})$, we get $\tilde{v}^l=0$ for every $l \geq j_0$. Therefore, we obtain 
\[
v_n(\cdot + x^{j_0}_n) \rightharpoonup V^{j_0}.
\]
The sequence $(x^{j_0}_n)_{n\geq 1}$ and the function $V^{j_0}$ now fulfill the conditions of Theorem $\ref{theorem compactness lemma intercritical NL4S}$. The proof is complete.
\defendproof 
\section{Global existence and blowup} \label{section global existence blowup}
We firstly use the sharp Gagliardo-Nirenberg inequality $(\ref{sharp gagliardo-nirenberg inequality intercritical 1})$ to show the following global existence. 
\begin{prop}[Global existence in $\dot{H}^{\Gace} \cap \dot{H}^2$] \label{prop global existence intercritical NL4S 1}
	Let $d\geq 5$ and $2_\star<\alpha<2^\star$. Let $u_0 \in \dot{H}^{\Gact} \cap \dot{H}^2$ and the corresponding solution $u$ to $(\ref{intercritical NL4S})$ defined on the maximal time $[0,T)$. Assume that
	\begin{align}
	\sup_{t\in [0,T)} \|u(t)\|_{\dot{H}^{\Gact}} < S_{\emph{gs}}. \label{assumption boundedness H dot gamma}
	\end{align}
	Then $T=\infty$, i.e. the solution exists globally in time.
\end{prop}
\begin{proof}
	By the sharp Gagliardo-Nirenberg inequality $(\ref{sharp gagliardo-nirenberg inequality intercritical 1})$, we bound
	\begin{align*}
	E(u(t)) &=\frac{1}{2} \|u(t)\|^2_{\dot{H}^2} -\frac{1}{\alpha+2} \|u(t)\|^{\alpha+2}_{L^{\alpha+2}} \\
	&\geq \frac{1}{2} \left( 1- \Big( \frac{\|u(t)\|_{\dot{H}^{\Gact}}}{S_{\text{gs}}}\Big)^\alpha\right) \|u(t)\|^2_{\dot{H}^2}.
	\end{align*}
	Thanks to the conservation of energy and the assumption $(\ref{assumption boundedness H dot gamma})$, we obtain $\sup_{t\in [0,T)} \|u(t)\|_{\dot{H}^2} <\infty$. By the blowup alternative given in Proposition $\ref{prop local well-posedness H dot 2}$ and $(\ref{assumption boundedness H dot gamma})$, the solution exists globally in time. The proof is complete.
\end{proof}
We also have the following global well-posedness result.
\begin{prop} \label{prop global existence intercritical NL4S 2} 
	Let $d\geq 5$ and $2_\star<\alpha<2^\star$. Let $u_0 \in \dot{H}^{\Gact} \cap \dot{H}^2$ and the corresponding solution $u$ to $(\ref{intercritical NL4S})$ defined on the maximal time $[0,T)$. Assume that
	\begin{align}
	S_{\emph{gs}} \leq \sup_{t\in [0,T)} \|u(t)\|_{\dot{H}^{\Gact}} < \infty, \quad \sup_{t\in [0,T)} \|u(t)\|_{L^{\alphce}} < L_{\emph{gs}}. \label{assumption boundedness H alpha}
	\end{align}
	Then $T=\infty$, i.e. the solution exists globally in time.
\end{prop}
The proof is similar to the one of Proposition $\ref{prop global existence intercritical NL4S 1}$ by using the shap Gagliardo-Nirenberg inequality $(\ref{sharp gagliardo-nirenberg inequality intercritical 2})$. \newline
\indent We next recall blowup criteria for $H^2$ solutions to the equation $(\ref{intercritical NL4S})$ due to \cite{BoulengerLenzmann}. 
\begin{prop} [Blowup in $H^2$ \cite{BoulengerLenzmann}] \label{prop blowup H2}
	Let $d \geq 2$, $ 2_\star<\alpha <2^\star$, $\alpha \leq 8$ and $u_0 \in H^2$ be radial. Assume that 
	\[
	E(u_0) M(u_0)^\sigma < E(Q) M(Q)^\sigma, \quad \|u_0\|_{\dot{H}^2} \|u_0\|^\sigma_{L^2} > \|Q\|_{\dot{H}^2} \|Q\|^\sigma_{L^2},
	\]
	where 
	\begin{align}
	\sigma :=\frac{2-\Gace}{\Gact} = \frac{8-(d-4)\alpha}{d\alpha-8}. \label{define sigma}
	\end{align}
	Then the corresponding solution $u$ to $(\ref{intercritical NL4S})$ blows up in finite time. 
\end{prop}
\begin{rem} \label{rem blowup H2}
	\begin{itemize}
		\item The restriction $\alpha \leq 8$ comes from the radial Sobolev embedding (or Strauss's inequality). An analogous restriction on $\alpha$ appears in the blowup of $H^1$ solutions for the nonlinear Schr\"odinger equation.
		\item Note that if $E(u_0)<0$, then the assumption $E(u_0) M(u_0)^\sigma < E(Q) M(Q)^\sigma$ holds trivially. 
	\end{itemize}
\end{rem}
If we assume $u_0 \in \dot{H}^{\Gact} \cap \dot{H}^2$, then the above blowup criteria does not hold due to the lack of mass conservation. Nevertheless, we have the following blowup criteria for initial data in $\dot{H}^{\Gact} \cap \dot{H}^2$. 
\begin{prop} [Blowup in $\dot{H}^{\Gace} \cap \dot{H}^2$] \label{prop blowup H dot 2}
	Let $d\geq 5$, $2_\star<\alpha <2^\star$, $\alpha< 4$ and $u_0 \in\dot{H}^{\Gace} \cap \dot{H}^2$ be radial satisfying $E(u_0)<0$. Assume that the corresponding solution $u$ to $(\ref{intercritical NL4S})$ defined on a maximal interval $[0,T)$ satisfies 
	\begin{align}
	\sup_{t\in [0,T)} \|u(t)\|_{\dot{H}^{\Gact}} <\infty. \label{uniform bounded assumption}
	\end{align}
	Then the solution $u$ to $(\ref{intercritical NL4S})$ blows up in finite time. 
\end{prop}
\begin{proof}
	Let $\theta: [0,\infty) \rightarrow [0,\infty)$ be a smooth function such that
	\[
	\theta(r) = \left\{  
	\begin{array}{cl}
	r^2 &\text{if } r \leq 1, \\
	0 &\text{if } r \geq 2, 
	\end{array}
	\right.
	\quad \text{and} \quad \theta''(r) \leq 2 \text{ for } r\geq 0. 
	\]
	For $R>0$ given, we define the radial function $\varphi_R: \R^d \rightarrow \R$ by 
	\begin{align}
	\varphi_R(x) = \varphi_R(r) := R^2 \theta(r/R), \quad |x|=r. 
	\end{align}
	By definition, we have
	\[
	2-\varphi''_R(r) \geq 0, \quad  2-\frac{\varphi'_R(r)}{r} \geq 0, \quad 2d -\Delta \varphi_R(x) \geq 0, \quad \text{for all } r \geq 0 \text{ and all } x \in \R^d,
	\]
	and
	\[
	\|\nabla^j \varphi_R \|_{L^\infty} \lesssim R^{2-j}, \quad \text{for } j=0, \cdots, 6,
	\]
	and also,
	\[
	\text{supp}(\nabla^j \varphi_R) \subset \left\{ 
	\begin{array}{cl}
	\{ |x| \leq 2R \} &\text{for } j=1,2, \\
	\{ R \leq |x| \leq 2R \} &\text{for } j=3, \cdots, 6.
	\end{array}
	\right.
	\]
	Let $u \in \dot{H}^{\Gact} \cap \dot{H}^2$ be a solution to $(\ref{intercritical NL4S})$. We define the localized virial action associated to $(\ref{intercritical NL4S})$ by
	\begin{align}
	M_{\varphi_R}(t) := 2 \int \nabla \varphi_R(x) \cdot \im{(\overline{u}(t, x) \nabla u(t,x))} dx. \label{virial action}
	\end{align}
	We firstly show that $M_{\varphi_R}(t)$ is well-defined. To do so, we need the following estimate
	\begin{align}
	\|u\|_{L^2(|x| \lesssim R)} \lesssim R^{\Gact} \|u\|_{L^{\alphct}(|x| \lesssim R)} \lesssim R^{\Gact} \|u\|_{\dot{H}^{\Gact}(|x| \lesssim R)}, \label{holder inequality}
	\end{align}
	which follows easily by H\"older's inequality and the Sobolev embedding. Here $\Gact$ and $\alphct$ are given in $(\ref{critical sobolev exponent})$ and $(\ref{critical lebesgue exponent})$ respectively. Since $\nabla \varphi_R$ is supported in $|x| \lesssim R$, the H\"older inequality together with $(\ref{holder inequality})$ imply
	\begin{align*}
	|M_{\varphi_R}(t)| &\lesssim \|\nabla \varphi_R\|_{L^\infty} \|u(t)\|_{L^2(|x| \lesssim R)} \|\nabla u(t)\|_{L^2(|x| \lesssim R)} \\
	&\lesssim \|\nabla \varphi_R\|_{L^\infty} \|u(t)\|^{3/2}_{L^2(|x| \lesssim R)} \|\Delta u(t)\|^{1/2}_{L^2(|x| \lesssim R)} \\
	&\lesssim R^{3\Gact/2} \|\nabla \varphi_R\|_{L^\infty} \|u(t)\|^{3/2}_{\dot{H}^{\Gact}(|x| \lesssim R)} \|u(t)\|^{1/2}_{\dot{H}^2(|x| \lesssim R)}.
	\end{align*}
	Note that in the case $\theta(r) = r^2$ and $\varphi_R(x)=|x|^2$, we have formally the virial law (see e.g. \cite{BoulengerLenzmann}):
	\begin{align}
	\begin{aligned}
	M'_{|x|^2}(t)=\frac{d}{dt} \Big( 4 \int x \cdot \im{( \overline{u}(t,x) \nabla u(t,x))} dx \Big) &= 16 \|\Delta u(t)\|^2_{L^2} - \frac{4d\alpha}{\alpha+2} \|u(t)\|^{\alpha+2}_{L^{\alpha+2}}  \\
	&= 4d\alpha E(u(t)) -2(d\alpha-8) \|\Delta u(t)\|^2_{L^2}. 
	\end{aligned}
	\label{virial law}
	\end{align}
	We have the following variation rate of the virial action (see e.g. \cite[Lemma 3.1]{BoulengerLenzmann} or \cite[Proposition 3.1]{MiaoWuZhang}):
	\begin{align}
	M'_{\varphi_R}(u(t)) &= \int \Delta^3 \varphi_R |u|^2 dx - 4 \sum_{j,k} \int \partial^2_{jk} \Delta \varphi_R \re{(\partial_j \overline{u} \partial_k u)} dx + 8 \sum_{j,k,l} \int \partial^2_{jk} \varphi_R \re{(\partial^2_{lj} \overline{u} \partial^2_{kl} u) } dx \nonumber \\
	& \mathrel{\phantom{=  \int \Delta^3 \varphi_R |u|^2 dx }} - 2 \int \Delta^2 \varphi_R |\nabla u|^2 dx -\frac{2\alpha}{\alpha+2} \int \Delta \varphi_R |u|^{\alpha+2} dx. \label{variation rate}
	\end{align}
	Since $\varphi_R(x) = |x|^2$ for $|x| \leq R$, we use $(\ref{virial law})$ to have
	\begin{align*}
	M'_{\varphi_R}(t)&= 16 \|\Delta u(t)\|^2_{L^2} -\frac{4d\alpha}{\alpha+2} \|u(t)\|^{\alpha+2}_{L^{\alpha+2}}  - 16 \|\Delta u(t)\|^2_{L^2(|x| >R)} +\frac{4d\alpha}{\alpha+2}\|u(t)\|^{\alpha+2}_{L^{\alpha+2}(|x|>R)} \\
	& \mathrel{\phantom{= 16 \|\Delta u(t)\|^2_{L^2}}} + \int_{|x|>R} \Delta^3 \varphi_R |u(t)|^2 dx - 4 \sum_{j,k} \int_{|x|>R} \partial^2_{jk} \Delta \varphi_R \re{(\partial_j \overline{u}(t) \partial_k u(t))} dx \\
	& \mathrel{\phantom{= 16 \|\Delta u(t)\|^2_{L^2}}} + 8 \sum_{j,k,l} \int_{|x|>R} \partial^2_{jk} \varphi_R \re{(\partial^2_{lj} \overline{u}(t) \partial^2_{kl} u(t)) } dx \\
	& \mathrel{\phantom{= 16 \|\Delta u(t)\|^2_{L^2} }} - 2 \int_{|x|>R} \Delta^2 \varphi_R |\nabla u(t)|^2 dx -\frac{2\alpha}{\alpha+2} \int_{|x|>R} \Delta \varphi_R |u(t)|^{\alpha+2} dx  \\	
	&= 4d\alpha E(u(t)) -2(d\alpha-8) \|\Delta u(t)\|^2_{L^2} + \int_{|x|>R} \Delta^3 \varphi_R |u(t)|^2 dx \\
	&\mathrel{\phantom{=}} - 4\sum_{j,k} \int_{|x|>R} \partial^2_{jk}\Delta \varphi_R \re{(\partial_j \overline{u}(t) \partial_k u(t) )} dx -2 \int_{|x|>R} \Delta^2 \varphi_R |\nabla u(t)|^2 dx 
	\end{align*}
	\begin{align*}
\mathrel{\phantom{M'_{\varphi_R}(t) = }}	&\mathrel{\phantom{=}} + 8 \sum_{j,k,l} \int_{|x|>R} \partial^2_{jk} \varphi_R \re{(\partial^2_{lj} \overline{u} (t) \partial^2_{kl} u(t) )} dx - 16 \|\Delta u(t)\|^2_{L^2(|x|>R)} \\
	&\mathrel{\phantom{= + 8 \sum_{j,k,l} \int_{|x|>R} \partial^2_{jk} \varphi_R \re{(\partial^2_{lj} \overline{u} (t) \partial^2_{kl} u(t) )} dx}} + \frac{2\alpha}{\alpha+2} \int_{|x|>R} (2d-\Delta \varphi_R) |u(t)|^{\alpha+2} dx.
	\end{align*}
	By the choice of $\varphi_R$, the assumption $(\ref{uniform bounded assumption})$ and $(\ref{holder inequality})$, we bound
	\begin{align*}
	\Big|\int_{|x|>R} \Delta^3 \varphi_R |u(t)|^2 dx \Big| &\lesssim R^{-4} \|u(t)\|^2_{L^2(|x| \lesssim R)} \lesssim R^{-2(2-\Gact)} \|u(t)\|^2_{\dot{H}^{\Gact}} \lesssim R^{-2(2-\Gact)}, \\
	\Big| \int_{|x|>R} \partial^2_{jk} \Delta \varphi_R \partial_j \overline{u} (t) \partial_k u (t) dx\Big|  &\lesssim R^{-2} \|\nabla u\|^2_{L^2(|x| \lesssim R)} \lesssim R^{-(2-\Gact)} \|u(t)\|_{\dot{H}^{\Gact}} \|\Delta u(t)\|_{L^2} \\
	&\mathrel{\phantom{\lesssim R^{-2} \|\nabla u\|^2_{L^2(|x| \lesssim R)}  \ }} \lesssim R^{-(2-\Gact)} \|\Delta u(t)\|_{L^2}, \\
	\Big| \int_{|x|>R} \Delta^2\varphi_R |\nabla u(t)|^2 dx \Big| &\lesssim R^{-(2-\Gact)} \|u(t)\|_{\dot{H}^{\Gact}} \|\Delta u(t)\|_{L^2} \lesssim R^{-(2-\Gact)} \|\Delta u(t)\|_{L^2}.
	\end{align*}
	Using the fact
	\[
	\partial^2_{jk} = \Big( \delta_{jk} - \frac{x_j x_k}{r^2} \Big) \frac{\partial_r}{r} + \frac{x_j x_k}{r^2} \partial^2_r,
	\]
	a calculation combined with integration by parts yields that
	\begin{align*}
	\sum_{j,k,l} \int \partial^2_{jk} \varphi_R \partial^2_{lj} \overline{u}(t) \partial^2_{kl} u(t) dx &= \int \varphi''_R |\partial^2_r u(t)|^2 + \frac{d-1}{r^2} \frac{\varphi'_R}{r} |\partial_r u(t)|^2 dx \\
	&= 2\int |\Delta u(t)|^2 - (2-\varphi''_R) |\partial^2_r u(t)|^2 - \Big( 2-\frac{\varphi'_R}{r}\Big) \frac{d-1}{r^2} |\partial_r u(t)|^2 dx \\
	& \leq 2\|\Delta u(t)\|^2_{L^2}.
	\end{align*} 
	Here we use the identity 
	\[
	\|\Delta u(t)\|^2_{L^2} = \int |\partial^2_r u(t)|^2 + \frac{d-1}{r^2} |\partial_r u(t)|^2 dx.
	\]
	Thus,
	\[
	8 \sum_{j,k,l} \int_{|x|>R} \partial^2_{jk} \varphi_R \re{(\partial^2_{lj} \overline{u} (t) \partial^2_{kl} u(t) )} dx - 16 \|\Delta u(t)\|^2_{L^2(|x|>R)} \leq 0.
	\]
	We obtain
	\begin{align*}
	M'_{\varphi_R}(t) &\leq 4d\alpha E(u(t)) -2(d\alpha-8) \|\Delta u(t)\|^2_{L^2} + O\Big( R^{-2(2-\Gact)} + R^{-(2-\Gact)} \|\Delta u(t)\|_{L^2} \Big) \\
	&\mathrel{\phantom{\leq 4d\alpha E(u(t)) -2(d\alpha-8) \|\Delta u(t)\|^2_{L^2}}}+ \frac{2\alpha}{\alpha+2} \int_{|x|>R} (2d-\Delta \varphi_R) |u(t)|^{\alpha+2} dx.
	\end{align*}
	We now estimate the last term of the above inequality. To do so, we use the argument of \cite{MerleRaphael}. Consider for $A>0$ the annulus $\mathcal{C} = \{A < |x| \leq 2A\}$, we claim that for any $\eps>0$,
	\begin{align}
	\|u(t)\|^{\alpha+2}_{L^{\alpha+2}(\mathcal{C})} \leq \eps \|\Delta u(t)\|_{L^2(\mathcal{C})} + C(\eps)A^{-2(2-\Gact)}. \label{estimate on annulus}
	\end{align}
	To see this, we use the radial Sobolev embedding (see e.g. \cite{Strauss}) and $(\ref{holder inequality})$ to estimate
	\begin{align*}
	\|u(t)\|^{\alpha+2}_{L^{\alpha+2}(\mathcal{C})} &\lesssim \Big( \sup_{\mathcal{C}} |u(t,x)| \Big)^\alpha \|u(t)\|^2_{L^2(\mathcal{C})} \\
	& \lesssim A^{-\frac{(d-1)\alpha}{2}} \|\nabla u(t)\|^{\frac{\alpha}{2}}_{L^2(\mathcal{C})} \|u(t)\|^{\frac{\alpha}{2}+2}_{L^2(\mathcal{C})} \\
	&\lesssim A^{-\frac{(d-1)\alpha}{2}} \|\Delta u(t)\|^{\frac{\alpha}{4}}_{L^2(\mathcal{C})} \|u(t)\|^{\frac{3\alpha}{4}+2}_{L^2(\mathcal{C})} \\
	&\lesssim A^{-\vartheta} \|\Delta u(t)\|^{\frac{\alpha}{4}}_{L^2(\mathcal{C})},
	\end{align*}
	where 
	\[
	\vartheta = \frac{(d-1)\alpha}{2}- \left(\frac{3\alpha}{4}+2\right) \Gact= 2(2-\Gact)\frac{4-\alpha}{4} > 0.
	\]
	By the Young inequality, we have for any $\eps>0$,
	\[
	\|u(t)\|^{\alpha+2}_{L^{\alpha+2}(\mathcal{C})} \lesssim  \eps \|\Delta u(t)\|_{L^2(\mathcal{C})} + \eps^{-\frac{\alpha}{4-\alpha}} A^{-\frac{4\vartheta}{4-\alpha}} = \eps \|\Delta u(t)\|_{L^2(\mathcal{C})} + C(\eps) A^{-2(2-\Gact)}.
	\]
	This shows the claim above. Note that the condition $\alpha<4$ is crucial to show $(\ref{estimate on annulus})$. We now write
	\begin{align*}
	\int_{|x|>R} |u(t)|^{\alpha+2} dx  = \sum_{j=0}^\infty \int_{2^j R<|x| \leq 2^{j+1} R} |u(t)|^{\alpha+2} dx, 
	\end{align*}
	and apply $(\ref{estimate on annulus})$ with $A=2^j R$ to get
	\begin{align*}
	\int_{|x|>R} |u(t)|^{\alpha+2} dx &\leq \eps \sum_{j=0}^\infty \|\Delta u(t)\|_{L^2(2^{j} R < |x| \leq 2^{j+1} R)} + C(\eps) \sum_{j=0}^\infty ( 2^j R)^{-2(2-\Gact)} \\
	& \leq \eps \|\Delta u(t)\|_{L^2(|x|>R)} + C(\eps) R^{-2(2-\Gact)}.
	\end{align*}
	Since $\|2d-\varphi_R\|_{L^\infty} \lesssim 1$, we obtain for any $\eps>0$,
	\[
	\int_{|x|>R} (2d-\varphi_R) |u(t)|^{\alpha+2} dx \lesssim \eps \|\Delta u(t)\|_{L^2(|x|>R)} + C(\eps) R^{-2(2-\Gact)}.
	\]
	Therefore, 
	\begin{align}
	M'_{\varphi_R}(t) \leq 4d \alpha E(u(t)) - 2(d\alpha-8) \|\Delta u(t)\|^2_{L^2} + O\Big(R^{-2(2-\Gact)} &+ R^{-(2-\Gact)} \|\Delta u(t)\|_{L^2} \label{estimate variation rate} \\
	& + \eps \|\Delta u(t)\|_{L^2} + C(\eps) R^{-2(2-\Gact)} \Big). \nonumber
	\end{align}
	By taking $\eps>0$ small enough and $R>0$ large enough depending on $\eps$, the conservation of energy implies
	\begin{align}
	M'_{\varphi_R}(t) \leq 2d\alpha E(u_0) - \delta \|\Delta u(t)\|^2_{L^2}, \label{negative variation rate}
	\end{align}
	for all $t\in [0,T)$, where $\delta:= d\alpha-8>0$. With $(\ref{negative variation rate})$ at hand, the finite time blowup follows by a standard argument (see e.g. \cite{BoulengerLenzmann}). 
\end{proof}
\section{Blowup concentration} \label{section blowup concentration}
\setcounter{equation}{0}
\begin{theorem} [Blowup concentration] \label{theorem H dot gamma concentration intercritical NL4S}
	Let $d\geq 5$ and $2_\star <\alpha<2^\star$. Let $u_0 \in \dot{H}^{\Gact} \cap \dot{H}^2$ be such that the corresponding solution $u$ to $(\ref{intercritical NL4S})$ blows up at finite time $0<T<\infty$. Assume that the solution satisfies 
	\begin{align}
	\sup_{t\in [0,T)} \|u(t)\|_{\dot{H}^{\Gact}} < \infty. \label{assumption blowup intercritical}
	\end{align}
	Let $a(t)>0$ be such that 
	\begin{align}
	a(t) \|u(t)\|_{\dot{H}^2}^{\frac{1}{2-\Gace}} \rightarrow \infty, \label{condition of a intercritical}
	\end{align}
	as $t\uparrow T$. Then there exist $x(t), y(t) \in \R^d$ such that
	\begin{align}
	\liminf_{t\uparrow T} \int_{|x-x(t)| \leq a(t)} |(-\Delta)^{\frac{\Gace}{2}}u(t,x)|^2 dx \geq S_{\emph{gs}}^2, \label{H dot gamma concentration intercritical}
	\end{align}
	and
	\begin{align}
	\liminf_{t\uparrow T} \int_{|x-y(t)| \leq a(t)} |u(t,x)|^{\alphce} dx \geq L_{\emph{gs}}^2. \label{L alpha concentration intercritical}
	\end{align}
\end{theorem}
\begin{rem} \label{rem H dot gamma concentration intercritical NL4S}
	\begin{itemize}
		\item The restriction $d\geq 5$ comes from the local well-posedness and blowup results. The result still holds true for dimensions $d \leq 4$ provided that one can show local well-posedness and blowup in such dimensions.
		\item By the blowup rate given in Corollary $\ref{coro blowup rate intercritical}$ and the assumption $(\ref{assumption blowup intercritical})$, we have 
		\[
		\|u(t)\|_{\dot{H}^2} > \frac{C}{(T-t)^{\frac{2-\Gact}{4}}},
		\]
		for $t\uparrow T$. Rewriting
		\begin{align*}
		\frac{1}{a(t) \|u(t)\|_{\dot{H}^2}^{\frac{1}{2-\Gace}}} = \frac{\sqrt[4]{T-t}}{a(t)} \frac{1}{\sqrt[4]{T-t}\|u(t)\|_{\dot{H}^2}^{\frac{1}{2-\Gace}}} &= \frac{\sqrt[4]{T-t}}{a(t)} \left(\frac{1}{(T-t)^{\frac{2-\Gact}{4}} \|u(t)\|_{\dot{H}^2} } \right)^{\frac{1}{2-\Gact}} \\
		&<C\frac{\sqrt[4]{T-t}}{a(t)},
		\end{align*}
		we see that any function $a(t)>0$ satisfying $\frac{\sqrt[4]{T-t}}{a(t)} \rightarrow 0$ as $t\uparrow T$ fulfills the conditions of Theorem $\ref{theorem H dot gamma concentration intercritical NL4S}$. 
	\end{itemize}
\end{rem}
\noindent \textit{Proof of Theorem $\ref{theorem H dot gamma concentration intercritical NL4S}$.}
Let $(t_n)_{n\geq 1}$ be a sequence such that $t_n \uparrow T$ and $g\in \mathcal{G}$. Set
\[
\lambda_n := \left(\frac{\|g\|_{\dot{H}^2}}{\|u(t_n)\|_{\dot{H}^2}}\right)^{\frac{1}{2-\Gact}}, \quad v_n(x):= \lambda_n^{\frac{4}{\alpha}} u(t_n, \lambda_n x). 
\]
By the blowup alternative and the assumption $(\ref{assumption blowup intercritical})$, we see that $\lambda_n \rightarrow 0$ as $n \rightarrow \infty$. Moreover, we have
\begin{align*}
\|v_n\|_{\dot{H}^{\Gact}} = \|u(t_n)\|_{\dot{H}^{\Gact}} <\infty,
\end{align*}
uniformly in $n$ and 
\[
\|v_n\|_{\dot{H}^2} = \lambda_n^{2-\Gact}\|u(t_n)\|_{\dot{H}^2}= \|g\|_{\dot{H}^2},
\] 
and
\[
E(v_n) = \lambda_n^{2(2-\Gact)} E(u(t_n)) = \lambda_n^{2(2-\Gact)} E(u_0) \rightarrow 0, \quad \text{as } n \rightarrow \infty. 
\]
This implies in particular that
\begin{align*}
\|v_n\|^{\alpha+2}_{L^{\alpha+2}} \rightarrow \frac{\alpha+2}{2} \|g\|^2_{\dot{H}^2}, \quad \text{as } n \rightarrow \infty. 
\end{align*}
The sequence $(v_n)_{n\geq 1}$ satisfies the conditions of Theorem $\ref{theorem compactness lemma intercritical NL4S}$ with 
\[
m^{\alpha+2} = \frac{\alpha+2}{2} \|g\|^2_{\dot{H}^2}, \quad M^2 = \|g\|^2_{\dot{H}^2}. 
\]
Therefore, there exists a sequence $(x_n)_{n\geq 1}$ in $\R^d$ such that up to a subsequence,
\[
v_n(\cdot + x_n)  = \lambda_n^{\frac{4}{\alpha}} u(t_n, \lambda_n \cdot + x_n) \rightharpoonup V \text{ weakly in } \dot{H}^{\Gact} \cap \dot{H}^2,
\]
as $n \rightarrow \infty$ with $\|V\|_{\dot{H}^{\Gact}} \geq  S_{\text{gs}}$. In particular, 
\[
(-\Delta)^{\frac{\Gact}{2}} v(\cdot + x_n) = \lambda_n^{\frac{d}{2}} [(-\Delta)^{\frac{\Gact}{2}} u](t_n, \lambda_n \cdot + x_n) \rightharpoonup (-\Delta)^{\frac{\Gact}{2}} V \text{ weakly in } L^2. 
\]
This implies for every $R>0$,
\[
\liminf_{n\rightarrow \infty} \int_{|x|\leq R} \lambda_n^{d}| [(-\Delta)^{\frac{\Gact}{2}} u](t_n, \lambda_n x + x_n)|^2 dx \geq \int_{|x|\leq R} |(-\Delta)^{\frac{\Gact}{2}} V(x)|^2 dx,
\]
or 
\[
\liminf_{n\rightarrow \infty} \int_{|x-x_n|\leq R\lambda_n} | [(-\Delta)^{\frac{\Gact}{2}} u](t_n, x)|^2 dx \geq \int_{|x|\leq R} |(-\Delta)^{\frac{\Gact}{2}} V(x)|^2 dx.
\]
In view of the assumption $\frac{a(t_n)}{\lambda_n} \rightarrow \infty$ as $n\rightarrow \infty$, we get
\[
\liminf_{n\rightarrow \infty} \sup_{y\in \R^d} \int_{|x-y|\leq a(t_n)} |(-\Delta)^{\frac{\Gact}{2}}u(t_n, x)|^2 dx \geq \int_{|x|\leq R} |(-\Delta)^{\frac{\Gact}{2}}V(x)|^2 dx,
\]
for every $R>0$, which means that
\[
\liminf_{n\rightarrow \infty} \sup_{y \in \R^d} \int_{|x-y|\leq a(t_n)} |(-\Delta)^{\frac{\Gact}{2}}u(t_n, x)|^2 dx \geq \int |(-\Delta)^{\frac{\Gact}{2}}V(x)|^2 dx \geq  S_{\text{gs}}^2.
\]
Since the sequence $(t_n)_{n\geq 1}$ is arbitrary, we infer that
\[
\liminf_{t\uparrow T} \sup_{y\in \R^d} \int_{|x-y|\leq a(t)} |(-\Delta)^{\frac{\Gact}{2}}u(t,x)|^2 dx \geq S_{\text{gs}}^2.
\]
But for every $t \in (0,T)$, the function $y\mapsto \int_{|x-y| \leq a(t)} |(-\Delta)^{\frac{\Gact}{2}}u(t,x)|^2 dx$ is continuous and goes to zero at infinity. As a result, we get
\[
\sup_{y\in \R^d} \int_{|x-y|\leq a(t)} |(-\Delta)^{\frac{\Gact}{2}}u(t,x)|^2 dx = \int_{|x-x(t)| \leq a(t)} |(-\Delta)^{\frac{\Gact}{2}}u(t,x)|^2 dx,
\]
for some $x(t) \in \R^d$. This shows $(\ref{H dot gamma concentration intercritical})$. The proof for $(\ref{L alpha concentration intercritical})$ is similar using Item 2 of Theorem $\ref{theorem compactness lemma intercritical NL4S}$. The proof is complete.
\defendproof
\section{Limiting profile with critical norms} \label{section limiting profile}
\setcounter{equation}{0}
Let us start with the following characterization of solution with critical norms.
\begin{lem} \label{lem characterization critical norm intercritical}
	Let $d\geq 1$ and $2_\star<\alpha<2^\star$. 
	\begin{itemize}
		\item If $u \in \dot{H}^{\Gact} \cap \dot{H}^2$ is such that $\|u\|_{\dot{H}^{\Gact}}=S_{\emph{gs}}$ and $E(u)=0$, then $u$ is of the form
		\[
		u(x) = e^{i\theta} \lambda^{\frac{4}{\alpha}} g(\lambda x + x_0),
		\]
		for some $g\in \mathcal{G}$, $\theta \in \R, \lambda>0$ and $x_0 \in \R^d$. 
		\item If $u \in L^{\alphce} \cap \dot{H}^2$ is such that $\|u\|_{L^{\alphce}}=L_{\emph{gs}}$ and $E(u)=0$, then $u$ is of the form
		\[
		u(x) = e^{i\vartheta} \mu^{\frac{4}{\alpha}} h(\mu x + y_0),
		\]
		for some $h\in \mathcal{H}$, $\vartheta \in \R, \mu>0$ and $y_0 \in \R^d$.
	\end{itemize}
\end{lem}
\begin{proof}
	We only prove Item 1, Item 2 is treated similarly. Since $E(u)=0$, we have
	\[
	\|u\|^2_{\dot{H}^2} = \frac{2}{\alpha+2} \|u\|^{\alpha+2}_{L^{\alpha+2}}.
	\]
	Thus
	\[
	H(u) = \frac{\|u\|^{\alpha+2}_{L^{\alpha+2}}}{\|u\|^{\alpha}_{\dot{H}^{\Gact}} \|u\|^2_{\dot{H}^2}} = \frac{\alpha+2}{2} \|u\|^{-\alpha}_{\dot{H}^{\Gact}} = \frac{\alpha+2}{2} S_{\text{gs}}^{-\alpha} = A_{\text{GN}}.
	\]
	This shows that $u$ is the maximizer of $H$. It follows from Proposition $\ref{prop variational structure ground state intercritical}$ that $u$ is of the form $u(x) = a g(\lambda x +x_0)$ for some $g\in \mathcal{G}$, $a \in \C^\star, \lambda>0$ and $x_0 \in \R^d$. On the other hand, since $\|u\|_{\dot{H}^{\Gact}} = S_{\text{gs}}=\|g\|_{\dot{H}^{\Gact}}$, we have $|a|= \lambda^{\frac{4}{\alpha}}$. This shows the result. 
\end{proof}
We now have the following limiting profile of blowup solutions with critical norms. 
\begin{theorem}[Limiting profile with critical norms] \label{theorem limiting profile critical norm intercritical}
	Let $d\geq 5$ and $2_\star <\alpha<2^\star$. Let $u_0 \in \dot{H}^{\Gact} \cap \dot{H}^2$ be such that the corresponding solution $u$ to $(\ref{intercritical NL4S})$ blows up at finite time $0<T<\infty$. 
	\begin{itemize}
		\item Assume that 
		\begin{align}
		\sup_{t\in [0,T)} \|u(t)\|_{\dot{H}^{\Gact}} = S_{\emph{gs}}. \label{assumption critical sobolev norm}
		\end{align}
		Then there exist $g \in \mathcal{G}$, $\theta(t)\in \R$, $\lambda(t)>0$ and $x(t) \in \R^d$ such that 
		\[
		e^{i\theta(t)}  \lambda^{\frac{4}{\alpha}}(t) u(t, \lambda(t) \cdot + x(t)) \rightarrow g \text{ strongly in } \dot{H}^{\Gact} \cap \dot{H}^2 \text{ as } t \uparrow T.
		\]
		\item Assume that 
		\begin{align}
		\sup_{t\in [0,T)} \|u(t)\|_{\dot{H}^{\Gact}} < \infty, \quad \sup_{t\in [0,T)} \|u(t)\|_{L^{\alphce}} = L_{\emph{gs}}. \label{assumption critical lebesgue norm}
		\end{align}
		Then there exist $h\in \mathcal{H}$, $\vartheta(t)\in \R$, $\mu(t)>0$ and $y(t) \in \R^d$ such that 
		\[
		e^{i\vartheta(t)}  \mu^{\frac{4}{\alpha}}(t) u(t, \mu(t) \cdot + y(t)) \rightarrow h \text{ strongly in } L^{\alphce} \cap \dot{H}^2 \text{ as } t \uparrow T.
		\]
	\end{itemize}
\end{theorem}
\begin{proof}
	We only give the proof for the first case, the second case is similar. We will show that for any $(t_n)_{n\geq 1}$ satisfying $t_n \uparrow T$, there exist a subsequence still denoted by $(t_n)_{n\geq 1}$, $g\in \mathcal{G}$, sequences of $\theta_n \in \R, \lambda_n>0$ and $x_n \in \R^d$ such that
	\begin{align}
	e^{it\theta_n} \lambda^{\frac{4}{\alpha}}_n u(t_n, \lambda_n \cdot + x_n) \rightarrow g \text{ strongly in } \dot{H}^{\Gact} \cap \dot{H}^2 \text{ as } n \rightarrow \infty. \label{limiting profile critical norm proof intercritical}
	\end{align}
	Let $(t_n)_{n\geq 1}$ be a sequence such that $t_n \uparrow T$. Set
	\[
	\lambda_n := \left(\frac{\|Q\|_{\dot{H}^2}}{\|u(t_n)\|_{\dot{H}^2}}\right)^{\frac{1}{2-\Gact}}, \quad v_n(x):= \lambda_n^{\frac{4}{\alpha}} u(t_n, \lambda_n x),
	\]
	where $Q$ is as in Proposition $\ref{prop variational structure ground state intercritical}$. By the blowup alternative and $(\ref{assumption critical sobolev norm})$, we see that $\lambda_n \rightarrow 0$ as $n \rightarrow \infty$. Moreover, we have
	\begin{align}
	\|v_n\|_{\dot{H}^{\Gact}} = \|u(t_n)\|_{\dot{H}^{\Gact}} \leq S_{\text{gs}}=\|Q\|_{\dot{H}^{\Gact}},  \label{property v_n intercritical}
	\end{align}
	and
	\begin{align}
	\|v_n\|_{\dot{H}^2} = \lambda_n^{2-\Gact} \|u(t_n)\|_{\dot{H}^2} = \| Q\|_{\dot{H}^2}, \label{property v_n intercritical 1}
	\end{align}
	and
	\[
	E(v_n) = \lambda_n^{2(2-\Gact)} E(u(t_n)) = \lambda_n^{2(2-\Gact)} E(u_0) \rightarrow 0, \quad \text{as } n \rightarrow \infty. 
	\]
	This yields in particular that
	\begin{align}
	\|v_n\|^{\alpha+2}_{L^{\alpha+2}} \rightarrow \frac{\alpha+2}{2} \|Q\|^2_{\dot{H}^2}, \quad \text{as } n \rightarrow \infty. \label{convergence v_n intercritical}
	\end{align}
	The sequence $(v_n)_{n\geq 1}$ satisfies the conditions of Theorem $\ref{theorem compactness lemma intercritical NL4S}$ with 
	\[
	m^{\alpha+2} = \frac{\alpha+2}{2} \|Q\|^2_{\dot{H}^2}, \quad M^2 = \|Q\|^2_{\dot{H}^2}. 
	\]
	Therefore, there exists a sequence $(x_n)_{n\geq 1}$ in $\R^d$ such that up to a subsequence,
	\[
	v_n(\cdot + x_n)  = \lambda_n^{\frac{4}{\alpha}} u(t_n, \lambda_n \cdot + x_n) \rightharpoonup V \text{ weakly in } \dot{H}^{\Gact} \cap \dot{H}^2,
	\]
	as $n \rightarrow \infty$ with $\|V\|_{\dot{H}^{\Gact}} \geq S_{\text{gs}}$. 
	Since $v_n(\cdot +x_n) \rightharpoonup V$ weakly in $\dot{H}^{\Gact}\cap \dot{H}^2$ as $n\rightarrow \infty$, the semi-continuity of weak convergence and $(\ref{property v_n intercritical})$ imply
	\[
	\|V\|_{\dot{H}^{\Gact}} \leq \liminf_{n\rightarrow \infty} \|v_n\|_{\dot{H}^{\Gact}} \leq S_{\text{gs}}. 
	\]
	This together with the fact $\|V\|_{\dot{H}^{\Gact}} \geq S_{\text{gs}}$ show that
	\begin{align}
	\|V\|_{\dot{H}^{\Gact}} = S_{\text{gs}} = \lim_{n\rightarrow \infty} \|v_n\|_{\dot{H}^{\Gact}}. \label{H dot gamma norm v_n intercritical}
	\end{align}
	Therefore
	\[
	v_n(\cdot +x_n) \rightarrow V \text{ strongly in } \dot{H}^{\Gact} \text{ as } n\rightarrow \infty. 
	\]
	On the other hand, the Gagliardo-Nirenberg inequality $(\ref{sharp gagliardo-nirenberg inequality intercritical 1})$ shows that $v_n(\cdot +x_n) \rightarrow V$ strongly in $L^{\alpha+2}$ as $n \rightarrow \infty$. Indeed, by $(\ref{property v_n intercritical 1})$,
	\begin{align*}
	\|v_n(\cdot +x_n) -V\|^{\alpha+2}_{L^{\alpha+2}} &\lesssim \|v_n(\cdot + x_n) - V\|^{\alpha}_{\dot{H}^{\Gact}} \|v_n(\cdot + x_n) - V\|^2_{\dot{H}^2} \\
	&\lesssim (\|Q\|_{\dot{H}^2} + \|V\|_{\dot{H}^2})^2 \|v_n(\cdot +x_n) -V\|^{\alpha}_{\dot{H}^{\Gact}} \rightarrow 0,
	\end{align*}
	as $n\rightarrow \infty$. Moreover, using $(\ref{convergence v_n intercritical})$ and $(\ref{H dot gamma norm v_n intercritical})$, the sharp Gagliardo-Nirenberg inequality $(\ref{sharp gagliardo-nirenberg inequality intercritical 1})$ yields
	\begin{align*}
	\|Q\|^2_{\dot{H}^2} = \frac{2}{\alpha+2} \lim_{n\rightarrow \infty} \|v_n\|^{\alpha+2}_{L^{\alpha+2}} = \frac{2}{\alpha+2} \|V\|^{\alpha+2}_{L^{\alpha+2}} \leq \Big(\frac{\|V\|_{\dot{H}^{\Gact}}}{S_{\text{gs}}} \Big)^{\alpha} \|V\|^2_{\dot{H}^2} = \|V\|^2_{\dot{H}^2},
	\end{align*}
	or $\|Q\|_{\dot{H}^2} \leq \|V\|_{\dot{H}^2}$. By the semi-continuity of weak convergence and $(\ref{property v_n intercritical 1})$, 
	\[
	\|V\|_{\dot{H}^2} \leq \liminf_{n\rightarrow \infty} \|v_n\|_{\dot{H}^2} = \|Q\|_{\dot{H}^2}. 
	\]
	Therefore,
	\begin{align}
	\|V\|_{\dot{H}^2} = \|Q\|_{\dot{H}^2} = \lim_{n\rightarrow \infty} \|v_n\|_{\dot{H}^2}.\label{H dot 1 norm v_n intercritical}
	\end{align}
	Combining $(\ref{H dot gamma norm v_n intercritical}), (\ref{H dot 1 norm v_n intercritical})$ and using the fact $v_n(\cdot + x_n) \rightharpoonup V$ weakly in $\dot{H}^{\Gact} \cap \dot{H}^2$, we conclude that 
	\[
	v_n(\cdot + x_n) \rightarrow V \text{ strongly in } \dot{H}^{\Gact} \cap \dot{H}^2 \text{ as } n\rightarrow \infty.
	\]
	In particular, we have
	\[
	E(V) = \lim_{n\rightarrow \infty} E(v_n) =0.
	\]
	This shows that there exists $V \in \dot{H}^{\Gact} \cap \dot{H}^2$ such that
	\[
	\|V\|_{\dot{H}^{\Gact}} = S_{\text{gs}}, \quad E(V) =0.
	\]
	By Lemma $\ref{lem characterization critical norm intercritical}$, there exists $g\in \mathcal{G}$ such that $V(x) = e^{i\theta} \lambda^{\frac{4}{\alpha}} g(\lambda x +x_0)$ for some $\theta \in \R, \lambda>0$ and $x_0 \in \R^d$. Thus
	\[
	v_n(\cdot + x_n) = \lambda_n^{\frac{4}{\alpha}} u(t_n, \lambda_n \cdot + x_n) \rightarrow V = e^{i\theta} \lambda^{\frac{4}{\alpha}} g(\lambda \cdot +x_0) \text{ strongly in } \dot{H}^{\Gact} \cap \dot{H}^2 \text{ as } n \rightarrow \infty.
	\]
	Redefining variables as
	\[
	\overline{\lambda}_n:= \lambda_n \lambda^{-1}, \quad \overline{x}_n:= \lambda_n \lambda^{-1} x_0 +x_n,
	\]
	we get
	\[
	e^{-i\theta} \overline{\lambda}^{\frac{4}{\alpha}}_n u(t_n, \overline{\lambda}_n \cdot + \overline{x}_n) \rightarrow g \text{ strongly in } \dot{H}^{\Gact} \cap \dot{H}^2 \text{ as } n\rightarrow \infty.
	\]
	This proves $(\ref{limiting profile critical norm proof intercritical})$ and the proof is complete.
\end{proof}

\section*{Acknowledgments}
The author would like to express his deep thanks to his wife - Uyen Cong for her encouragement and support. He would like to thank his supervisor Prof. Jean-Marc Bouclet for the kind guidance and constant encouragement. He also would like to thank the reviewer for his/her helpful comments and suggestions. 


\end{document}